\documentclass{article}
\usepackage[%
journal=   
lang=american,   
]{ems-journal}

\newtheorem{theorem}{\textbf Theorem}[section]
\newtheorem{lemma}[theorem]{\textbf Lemma}

\newtheorem{corollary}[theorem]{\textbf Corollary}
\newtheorem{example}[theorem]{\textbf Example}
\newtheorem{definition}[theorem]{\textbf Definition}

\newtheorem{remark}[theorem]{\textbf Remark}
\newtheorem{assumption}[theorem]{\textbf Assumption}
\newcommand{\dist}{\operatorname{dist}}
\newcommand{\diam}{\operatorname{diam}}
\newcommand{\X}{\mathcal X}
\newcommand{\bR}{\mathbb{R}}

\def \bA {{\boldsymbol A}}
\def \bB {{\boldsymbol B}}
\def \bC {{\boldsymbol C}}
\def \bD {{\boldsymbol D}}
\def \bE {{\boldsymbol E}}
\def \bI {{\boldsymbol I}}
\def \bP {{\boldsymbol P}}
\def \bQ {{\boldsymbol Q}}
\def \bH {{\boldsymbol H}}
\def \ba {{\boldsymbol a}}
\def \bb {{\boldsymbol b}}
\def \bc {{\boldsymbol c}}
\def \bq {{\boldsymbol q}}

\def \bv {{\boldsymbol v}}
\newtheorem{theoremA}{Theorem}

\numberwithin{equation}{section}

\begin{document}

\title{Smallest Singular Value Estimates for Nonuniform Fourier Matrices via Periodic Nonuniform Sampling\footnote{*Corresponding author: Rongrong Lin. E-mail: linrr@gdut.edu.cn}}
\titlemark{The Smallest Singular Value of Nonuniform Fourier Matrices}


%

\emsauthor{1}{
	\givenname{Liang}
	\surname{Chen}
	\orcid{0000-0003-3750-1071}}{L.~Chen}
\emsauthor{2}{
	\givenname{Rongrong}
	\surname{Lin*}
	\orcid{0000-0002-6234-2183}}{R.~Lin}
\emsauthor{3}{
	\givenname{Haizhang}
	\surname{Zhang}
	\orcid{0000-0002-8241-3145}}{H.~Zhang}

\Emsaffil{1}{
	\department{Department of Mathematics}
	\organisation{Jiujiang University}
	\zip{332000}
	\city{Jiujiang}
	\country{P. R. China}
	\affemail{chenliang3@alumni.sysu.edu.cn}}
\Emsaffil{2}{
	\department{School of Mathematics and Statistics}
	\organisation{Guangdong University of Technology}
	\zip{510520}
	\city{Guangzhou}
	\country{P. R. China} 
	\affemail{linrr@gdut.edu.cn}}
\Emsaffil{3}{
	\department{School of Mathematics (Zhuhai)}
	\organisation{Sun Yat-sen University}
	\zip{519082}
	\city{Zhuhai}
	\country{P. R. China}
	\affemail{zhhaizh2@sysu.edu.cn}}

\classification[15A60, 42A15]{15A12}


\begin{abstract}
We study the smallest singular value of nonuniform Fourier matrices in two settings: clustered nodes and perturbations of an equispaced grid. By reducing the problem to spectral norm estimates for periodic nonuniform interpolation matrices, we obtain nearly optimal bounds in both cases. For clustered nodes, we derive the first local separation condition in which each required gap depends only on the sizes of the two neighboring clusters. For perturbations with the bound \(1/4\leq L<1/2\), our result confirms the conjecture of Austin and Trefethen on the \(2\)-norm Lebesgue constant up to a logarithmic factor.
\end{abstract}
\maketitle




\section{Introduction}
Fourier matrices with nonuniform nodes arise naturally in super-resolution \cite{batenkov2023super,li2021stable,li2022stability}, trigonometric interpolation \cite{Austin2016,austin2023trigonometric,austin2017,yu2023on}, nonuniform discrete Fourier transform \cite{dutt1993fast,Gelb2014}\cite[Chapter 7]{plonka2023numerical}, and sampling discretization \cite{ChuiShenZhong1993,chui1999polynomial,MarzoSeip2009}. In this paper, we consider nonuniform Fourier matrices of the form
\begin{equation}\label{EqAX}
\bA_{\X}
=
\Big[
    \exp(-2\pi i j x_k/M)
    :
    j\in \mathbb Z_M,\ k\in\mathbb Z_s
\Big],
\end{equation}
where $i$ is the imaginary unit,
\(
\mathbb Z_M:=\{0,1,\ldots,M-1\}
\) and \(
\mathcal X=\{x_k\mid k\in\mathbb Z_s\}\subset \mathbb R/M\mathbb Z=[0,M).
\)
When the nodes are equispaced and $s=N$, $A_{\X}$ in \eqref{EqAX} reduces to the classical discrete Fourier transform matrix, whose columns are mutually orthogonal and whose singular values are all identically equal to $\sqrt{M}$.

Estimating the singular values of nonuniform Fourier matrices is fundamental to analyzing the stability of super-resolution algorithms and has attracted substantial attention in recent years \cite{aubel2019vandermonde,batenkov2023super,batenkov2021single,li2020super,shah2026}. Let \(\Delta\) denote the minimum wrap-around separation between the normalized nodes \(x_k/M\) in the unit circle. Moitra \cite[Theorems 1.1 and 1.3]{Moitra2015} established a phase transition phenomenon: if \(\Delta>1/(M-1)\), then the condition number of \(A_{\mathcal X}\) in \eqref{EqAX} is bounded above by
\[
   \sqrt{(M+1/\Delta-1)/(M-1/\Delta-1)}.
\]
In contrast, if \(M=(1-\epsilon)/\Delta\) with \(0<\epsilon<1\) and $s=\Omega(M)$, there exists a set $\X$ such that the condition number of $A_{\X}$ scales as \(2^{\Omega(\epsilon s)}\), where $\Omega$ denotes the big-Omega symbol.  In addition, Barnett \cite{barnett2022how} estimated the condition number of a contiguous submatrix of the Fourier matrix and proved that this quantity grows exponentially, with an explicit constant inside the exponent.

Inspired by the Beurling-Selberg
majorant and minorant \cite{Selberg1991,MontgomeryVaughan1974}, the proof strategy from \cite{Moitra2015} constructed entire functions with favorable time-frequency localization as window functions for one-sided approximating the indicator function of an interval. 
Later, Aubel and B\"{o}lcskei \cite[(34)]{aubel2019vandermonde} weakened the threshold condition \(\Delta>1/(M-1)\) from Theorem 1.1 of \cite{Moitra2015} to \(\Delta > 1/M\) and derived a sharper upper bound \[\sqrt{(M+1/\Delta-1)/(M-1/\Delta)}.\]
These approaches estimate extremal singular values through Beurling–Selberg type majorants or minorants and the associated quadratic forms of Fourier Gram matrices.

The method developed here follows a different route. We do not construct Beurling–Selberg type majorants or
minorants for an interval indicator. Instead, we embed the Fourier matrix of interest into a square nonuniform
Fourier matrix and represent the relevant part of its inverse by periodic Lagrange interpolation functions. This reduces the smallest singular value problem to a spectral-norm estimate for a periodic nonuniform interpolation matrix. For clustered nodes, this reduction is combined with a new construction of near-integer auxiliary nodes and compensation points adapted to individual clusters. The local character of this construction is precisely what allows the intercluster separation requirement to depend on the sizes of neighboring clusters rather than on the largest cluster in the entire configuration.

\textbf{Clustered nodes.} The separation threshold \(\Delta> 1/M\), however, excludes an important super-resolution regime where some frequencies are spaced closer than the Rayleigh length. A commonly examined approach to represent this scenario is the cluster-separation model: nodes inside the same cluster can be closer than $1/M$, whereas different clusters are clearly separated \cite{batenkov2021spectral,BatenkovGoldmanYomdin2021,batenkov2020conditioning,batenkov2021single,li2025multiscale,li2025new}. Our first contribution (see Theorem \ref{T001}) is to establish a novel lower bound for the smallest singular value of nonuniform Fourier matrices equipped with cluster-structured frequencies.

To facilitate the subsequent analysis, we first introduce the necessary notation and assumptions. For any matrix $\bB$, the spectral norm 
\begin{equation}\label{eqspectral}
\|\bB\|_2:=\max_{\|\ba\|_2=1}\|\bB\ba\|_2
\end{equation}
is its largest singular value \cite[Definition 5.6.1 and Example 5.6.6]{Horn2012}, where $\|\ba\|_2$ denotes the standard Euclidean norm of the vector $\ba$. When $\bB$ is invertible, $\bB^{-1}$ denotes its inverse, and $\|\bB^{-1}\|_2=1/\sigma_{\min}(\bB)$ is the reciprocal of $\bB$'s smallest singular value $\sigma_{\min}(\bB)$. For a vector $\bv\in\bR^n$, let ${\textrm{diag}}(\bv)$ denote the $n\times n$ diagonal matrix with $v_j$ as its $j$-th diagonal entry. For \( a \in \mathbb{R} \) and $b\ne0$, the expression \(r= a \pmod{b}\) denotes
\(a - r \in b\mathbb{Z}\).
Given two values $a,b\in\mathbb{R}$, their distance modulo $M$ is defined by
\[
\operatorname{dist}_M(a,b):=\min_{n\in\mathbb{Z}}|a-b-nM|.
\]
For any subset $T_1,T_2\subset\mathbb{R}$ and $x\in\mathbb{R}$, we define 
\[
\dist_M(x,T_1):=\min_{a\in T_1}\dist_M(x,a), \quad \dist_M(T_1,T_2):=\min_{a\in T_1,b\in T_2}\dist_M(a,b),
\]
and
\[\diam_M(T_1):=\max_{a,b\in T_1}\dist_M(a,b).
\]
The cardinality of a subset $S\subset\mathbb{R}$ is denoted by $|S|$. For notational convenience, we write the asymptotic notation $g_1(t)\lesssim g_2(t)$ (respectively, $g_1(t)\gtrsim g_2(t)$) whenever there exists an absolute constant $C>0$ such that $g_1(t)\le C g_2(t)$ (respectively, $g_1(t)\ge C g_2(t)$ ) holds for all $t\in T$, where $T$ denotes the domain of the variable $t$. We further write $ g_1(t)\asymp g_2(t)$ to indicate that both $ g_1(t)\lesssim g_2(t)$ and $ g_2(t)\lesssim g_1(t)$ are valid.

\begin{assumption}\label{Assumption} Suppose that the nonempty set  $\X=\{x_i\mid i\in\mathbb{Z}_s\}\subset \mathbb{R} / M\mathbb{Z}=[0,M)$ can be decomposed into $r\ge2$ separated and disjoint clusters,
\[
\X = \mathcal{C}_1\cup \mathcal{C}_2\cup \cdots \cup \mathcal{C}_r,
\]
which are indexed in cyclic order on $\mathbb{R}/M\mathbb{Z}$, with $\mathcal{C}_{r+1}:=\mathcal{C}_1$.
Define $\alpha_{q}:=\diam_{M}(\mathcal{C}_q)$ for any $p,q\in\{1,2,\dots,r\}$. 
Assume that
\[ 
    \min_{1\le i\neq j\le s}\operatorname{dist}_M(x_i,x_j)\ge M\delta, \quad \delta\le 1/M, \mbox{ and }\max_{1\le q\le r}|\mathcal{C}_q|=\kappa\ge2.
\]
Moreover, whenever \(r\ge2\), we assume that
\[
\dist_{M}(\mathcal{C}_p,\mathcal{C}_q)=:\beta_{p,q}>\alpha_{p,q}:=\max\{\alpha_{p},\alpha_{q}\}\quad\text{for any distinct} ~p,q\in\{1,2,\dots,r\}
\]
and
\[
\min_{1\le p\neq q\le r}\dist_{M}(\mathcal{C}_p,\mathcal{C}_q):=\beta>0.
\]
\end{assumption}

In Assumption \ref{Assumption},  \(\kappa\) denotes the largest cluster size, \(\delta\) is the minimal separation within the whole set of points, \(\alpha\) is an upper bound for the diameters of the clusters, and \(\beta\) measures the separation between distinct clusters.

Our first contribution can be summarized as follows.
\begin{theorem}\label{T001} Let $M\ge 4s$, and \(\mathcal X:=\{x_k\mid k\in\mathbb Z_s\}\) be a subset of \(\mathbb R/M\mathbb Z= [0,M)\) satisfying Assumption \ref{Assumption}.  If  
\[
\beta_{p,p+1}\ge  |\mathcal{C}_p|+|\mathcal{C}_{p+1}|+2\mbox{ for all } p\in \{1,2,\dots,r\},
\] 
then the smallest singular value of $A_{\X}$ in \eqref{EqAX} is larger than $\sqrt{M}(CM\delta)^{\kappa-1}$, where $C>0$ is a constant.
\end{theorem}

For separated clusters, the case  \(\kappa=1\) is covered by the nearly optimal results in \cite{aubel2019vandermonde,Moitra2015}, so the assumption
\(\kappa\geq 2\) is natural. Table \ref{tab:comparison-separated-clumps} summarizes representative bounds from the literature.
As shown in \cite{li2025new}, all estimates displayed in Table \ref{tab:comparison-separated-clumps} have the optimal dependence on \(M\) and \(\delta\), but their principal distinction lies in the separability assumptions. Our separation requirement is
weaker than those in previous works. More importantly, our separation condition is {\em{local}}. To the best of our knowledge, this is the first smallest singular value bound for multi-cluster nonuniform Fourier matrices whose separation condition is pair-adaptive. That is, the required gap between two neighboring clusters depends only on the cardinalities of these two clusters, rather than on the total number of nodes, the largest cluster size in the entire configuration, or a uniform global separation parameter.

\begin{table}[htbp]
\centering
\caption{Comparison of existing lower bounds for separated clusters.}
\label{tab:comparison-separated-clumps}
\renewcommand{\arraystretch}{1}
\resizebox{\linewidth}{!}{
\begin{tabular}{llll}
\hline
\textbf{References} & \textbf{Separation requirement} & \textbf{Lower-bound order} & {\bf Local/Nonlocal}\\ \hline
Li-Liao \cite{li2021stable} 
&
 $\beta\gtrsim M(M-1)^{-3/2}s \kappa^{5/2}/\sqrt{\sigma}$
&
$\sqrt M (C_1M\delta)^{\kappa-1}$ &Nonlocal\\ 
Batenkov-Goldman \cite{batenkov2021single}
&
 $\beta\gtrsim_{\kappa} s$
&
$\sqrt M (C_2M\delta)^{\kappa-1}$&Nonlocal\\
Li \cite{li2025multiscale}
&
$\beta\ge 3\kappa$
&
$\sqrt{M/s}(C_3M\delta)^{\kappa-1}$&Nonlocal\\
{\textbf{This paper}}
&
$\beta_{p,p+1}\ge (|\mathcal{C}_p|+|\mathcal{C}_{p+1}|+2)$
&
$\sqrt{M}(C_4M\delta)^{\kappa-1}$&Local\\ \hline
\end{tabular}}
\end{table}

\textbf{Perturbations of an equispaced grid.} We next consider the square case \(\X := \{x_k\in\mathbb{R}\mid  k\in\mathbb{Z}_M\}\) with 
\[
\max_{k\in\mathbb{Z}_M} |x_k - k| =: L < \frac{1}{4},
\]
which has been studied in the literature \cite{asipchuk2025concerning,yu2023on}. The estimation of the singular values of the matrix \(A_{\X}\) in this setting is equivalent to the Marcinkiewicz–Zygmund inequality for trigonometric polynomials in the $L^2$ sense. Several classic studies \cite{ChuiShenZhong1993,chui1999polynomial,MarzoSeip2009} show that the threshold \(1/4\) ensures that the smallest singular value is bounded below by a constant independent of \(M\).  This statement can be viewed as a finite version of Kadec's $1/4$ theorem \cite{Kad64,young2001}.
Asipchuk et al. \cite{asipchuk2025concerning} estimated the singular values of nonuniform Fourier matrices in higher-dimensional settings under perturbations smaller than 1/4.
By contrast, the behavior of the condition number of the matrix $A_{\X}$ for the perturbation parameter $L\ge 1/4$  remains an open problem.

Austin and Trefethen \cite{austin2017} noted that although Kadec's 1/4 theorem in sampling theory suggests that the perturbation bound should not exceed 1/4, this threshold ceases to be critical for trigonometric polynomial approximation when the target function possesses a certain degree of smoothness. For any positive real number \(a\), write \(a = v + \gamma\), where \(v\) is a nonnegative integer and \(\gamma \in (0,1]\). We say that \(f\) has \(a\) derivatives if \(f\) is \(v\) times continuously differentiable and its \(v\)-th derivative is H\"{o}lder continuous with exponent \(\gamma\).  Austin and Trefethen established the following result: 

\begin{theoremA}\cite[Theorem 1]{austin2017}
For any \(L \in (0,\frac{1}{2})\) and the points $\tilde{x}_k=(k+s_k)\frac{2\pi}{2N+1}$ with $|s_k|\le L$ for all $|k|\le N$. If \(f\) has \(a > 4L\) derivatives, then
\[
|I - \tilde{I}_N|\lesssim N^{4L- a}\mbox{ and } \sup_{x\in[-\pi,\pi]}| f(x) - \tilde{t}_N(x)| \lesssim N^{4L- a},
\]
where $\tilde{t}_N(x)=\sum_{|k|\le N}c_ke^{ik x}$ denotes the unique degree $N$ trigonometric interpolant such that $\tilde{t}_N(\tilde{x}_k)=f(\tilde{x}_k)$ for all $|k|\le N$ and $ \tilde{I}_N=\int_{-\pi}^{\pi}\tilde{t}_N(x)dx$ is the corresponding quadrature approximation to $I=\int_{-\pi}^{\pi}f(x)dx$.
\end{theoremA}
Austin and Trefethen \cite{austin2017} conjectured that if the deviation of the sampling points from an equispaced grid is bounded by \(L < 1/2\), then, provided that the periodic function is more than \(2L\)-differentiable but not necessarily more than \(4L\)-differentiable, the trigonometric interpolant and quadrature still converge.  For the $2$-norm Lebesgue constant $\Gamma_{2,N}$ defined below, Austin and Trefethen also conjectured that $\Gamma_{2,N}\lesssim N^{4L-1}$ (see \cite[page 2119]{austin2017} or \cite[Conjecture 3.10]{Austin2016}).

\begin{definition}\label{def}\cite[Subsection 3.4.1]{Austin2016} The $2$-norm Lebesgue constant $\Gamma_{2,N}$ with respect to a set $\{\tilde{x}_k\mid  |k|\le N\}$ of $2N+1$ sampling points for $f:[-\pi,\pi]\to\mathbb{R}$ is defined as:
\[
\Gamma_{2,N}:=\sup_{\|{\boldsymbol f}\|_{\ell_{2,N}}\le1}\|\tilde{t}_N\|_{L^{2}[-\pi,\pi]},
\]
where $\tilde{t}_N(x):=\sum_{|k|\le N}c_ke^{ik x}$ denotes the unique degree $N$ trigonometric interpolant such that $\tilde{t}_N(\tilde{x}_k)=f(\tilde{x}_k)$ for all $|k|\le N$, ${\boldsymbol f}:=(f(x_k): |k|\le N)\in\mathbb{R}^{2N+1}$, and the normalized Euclidean norm $\|{\boldsymbol f}\|_{\ell_{2,N}}:=\|{\boldsymbol f}\|_2/\sqrt{2N+1}$.
\end{definition}

Yu and Townsend \cite[Theorem 2.1]{yu2023on} investigated this conjecture and derived an explicit lower bound of the smallest singular value under the restriction $L<1/4$. Moreover, under an additional oversampling assumption, they established that such a uniform lower bound (independent of $M$) persists for all  $1/4 \le  L< 1/2$ \cite[Subsection 6.2]{yu2023on}. So far, it remains unknown whether the conjectures of Austin and Trefethen hold within the new regime $1/4\le L<1/2$,  in contrast to the extensively examined regime $L< 1/4$.

Let $L$ be an arbitrary positive number.
Let $M\ge2$ be a positive integer and let the elements of the set $\Lambda$ satisfy 
\begin{equation}\label{AssumptionLambda}
\min_{\substack{j,k\in\mathbb{Z}_M,j\ne k}}\operatorname{dist}_M(\lambda_j,\lambda_k)\ge\sigma>0\mbox{ and }\max_{k\in\mathbb{Z}_M}\operatorname{dist}_M(\lambda_k,k)<L.
\end{equation}
Define an $M\times M$ Fourier matrix with respect to the set $\Lambda$ satisfying \eqref{AssumptionLambda} as follows:
\begin{equation}\label{Fouriermatrix}
\bA_{\sigma,L} = \Big[ \exp(-2\pi i j \lambda_k/M) : j,k\in{\mathbb Z}_M\Big].
\end{equation}

We present our second main contribution: the estimation of upper and lower bounds for $\bA_{\sigma,L}^{-1}$. The complete proofs are elaborated in Section \ref{Section3} and Appendix D, respectively.

\begin{theorem}\label{Theorem} Let $1/4 \le L<1/2$, and $\bA_{\sigma,L}$ be defined as in \eqref{Fouriermatrix}. Then,
\[
\|\bA_{\sigma,L}^{-1}\|_{2}\lesssim \frac{M^{4L-\frac{3}{2}}\log M}{1-2L}.
\]
\end{theorem}

\begin{theorem}\label{Thmlowerbound}  Let $1/4 \le L<1/2$, \(\Lambda=\{\lambda_k\in[-1/2,M-1/2]\mid  k\in\mathbb{Z}_M\}\) with 
\begin{equation}\label{eqspecifiedLambda}
\lambda_k=
\begin{cases}
k+L, &\mbox{ if } 0\le k\le N,\\
k-L, & \mbox{ if }N<k\le 2N,
\end{cases}
\end{equation}
and $\bA_{\sigma,L}$ be defined as in \eqref{Fouriermatrix}. Then,
\[
\|\bA_{\sigma,L}^{-1}\|_{2}\gtrsim \frac{M^{4L-\frac{3}{2}}}{1-2L}.
\]
\end{theorem}

The upper bound established in Theorem \ref{Theorem} is nearly optimal with respect to both \(M\)  and \(1-2L\), where the term \(1-2L\) corresponds to the minimal separation between distinct frequency points. Recalling Definition \ref{def} regarding the $2$-norm Lebesgue constant, we immediately obtain
\begin{equation}\label{Gamma2N}
\Gamma_{2,N}=\sup_{\|{\boldsymbol f}\|_{\ell_{2,N}}\le1}\|\bc\|_2=\sup_{\|{\boldsymbol f}\|_{\ell_{2,N}}\le1}\|(\bA^{-1}_L)^*{\boldsymbol f}\|_2=\sqrt{2N+1}\|\bA_{L}^{-1}\|_2,
\end{equation}
where $\bc:=(c_k:|k|\le N)$, the matrix $\bA_L$ is specified by
\begin{equation}\label{DefAL}
\bA_L:=[\exp(-i j\tilde{x}_k):-N\le k,j\le N],
\end{equation}
and perturbed grid points $\tilde{x}_k$ take the form $\tilde{x}_k:=(k+s_k)\frac{2\pi }{2N+1}$ with $|s_k|\le L$ for all $|k|\le N$. 
Let \(M=2N+1\) and \(\lambda_{N+k}:=N+k+s_k\) for \(|k|\le N\). Then
\[
\frac{2\pi \lambda_{k+N}}{M}=\frac{2\pi \lambda_{k+N}}{2N+1}=\frac{2\pi (N+k+s_k)}{2N+1}= \frac{2\pi N}{2N+1}+\tilde{x}_k.
\]
Let $\bP_1={\textrm{diag}}(\exp(-2\pi ijN/M): |j|\le N)$ and $\bP_2={\textrm{diag}}(
\exp(-2\pi iN(N+k+s_k)/M):|k|\le N)$. Then
\[\Big[ \exp\Big(\frac{-2\pi i (j+N) \lambda_{N+k}}{M}\Big) : |j|,|k|\le N\Big]=\bP_1\bA_{L}\bP_2, \]
and, further, $\|\bA_{L}^{-1}\|_{2}=\|\bA_{\sigma,L}^{-1}\|_{2}$. 
Consequently, Theorem \ref{Theorem} delivers the upper bound
\begin{equation}\label{AL}
\|\bA_{L}^{-1}\|_{2}=\|\bA_{\sigma,L}^{-1}\|_{2}\lesssim \frac{N^{4L-\frac{3}{2}}\log M}{1-2L}.
\end{equation}
Combining \eqref{Gamma2N} and \eqref{AL} yields
\begin{equation}\label{upperboundeq}
\Gamma_{2,N}=\sqrt{2N+1}\|\bA_{L}^{-1}\|_2=\sqrt{2N+1}\|\bA_{\sigma,L}^{-1}\|_2\lesssim \frac{N^{4L-1}\log M}{1-2L}.
\end{equation}
This result nearly confirms the conjecture put forward by Austin and Trefethen \cite[Conjecture 3.10]{Austin2016}, \cite[page 2119]{austin2017} regarding the upper bound of the \(2\)-norm Lebesgue constant $\Gamma_{2,N}\lesssim N^{4L-1}$ over the parameter regime \(1/4 \le L < 1/2\). 
Combining Theorem \ref{Thmlowerbound} and \eqref{Gamma2N}, there exists a  set \(\{\tilde{x}_k\mid  |k|\le N\}\subset[-\pi,\pi]\) satisfying $s_k=L$ for any $-N\le k\le 0$ and $s_k=-L$ for any $0<k\le N$ associated with the prescribed frequency set $\Lambda$ specified in \eqref{eqspecifiedLambda} such that   
\[
\Gamma_{2,N}=\sqrt{2N+1}\|\bA_{L}^{-1}\|_2= \sqrt{2N+1}\|\bA_{\sigma,L}^{-1}\|_2 \gtrsim  \frac{N^{4L-1}}{1-2L}. 
\]
The upper estimate \eqref{upperboundeq} and the matching lower bound demonstrate that the bound for \(\Gamma_{2,N}\) derived herein is nearly tight. Consequently, our theoretical analysis nearly confirms that these points $\lambda_k$, defined as in \eqref{eqspecifiedLambda}, are the maximizers stated in Conjecture 3.3 of \cite{Austin2016}.

Theorem \ref{Theorem} yields the following corollary, which also provides an affirmative answer to the conjecture of Austin and Trefethen \cite[Page 2115]{austin2017} on trigonometric quadrature (since $4L- 1<2L$ for any $1/4\le L<1/2$). Its proof is given at the end of Section \ref{Section3}. Our present method does not resolve the conjecture of Austin and Trefethen concerning trigonometric interpolation in the \(L^\infty\) sense. Related \(L^\infty\)-stability results for suitably perturbed Chebyshev-Lobatto points were recently obtained by Wu \cite{Wu2026PerturbedChebyshev}.

\begin{corollary}\label{coro} All notation are the same as those in Theorem A. Then
\[
 \|f- \tilde{t}_N\|_{L^{2}([-\pi,\pi])}\lesssim \frac{N^{4L- a-1}\log M}{1-2L} ,
 \mbox{ and }
 |I - \tilde{I}_N|\lesssim \frac{N^{4L- a-1}\log M}{1-2L}.
\]
\end{corollary}

Section \ref{Section2} formulates the general interpolation framework. We introduce the periodic nonuniform interpolation matrix $\bB_{\sigma,L}$ (defined in \eqref{DefeqBsigmaL}) and establish an exact equivalence between its spectral norm and the associated restricted inverse norm of the Fourier matrix. Section \ref{Section3} integrates this framework with a near-integer comparison configuration to prove Theorem \ref{Theorem}. Section \ref{Section4} develops the cluster-adapted auxiliary-node construction and furnishes the proof of Theorem \ref{T001}. Section \ref{Section5} concludes the paper with final remarks. Finally, we clarify the role of AI in paper preparation. All proofs pertaining to identity \eqref{gs0pre}, Lemmas \ref{Cnorm}, \ref{ExampleCnorm}, and \ref{L25}, as well as Theorem \ref{Thmlowerbound}, are deferred to the appendix.

\section{Spectral norms of periodic nonuniform interpolation matrices}\label{Section2}

\subsection{Preliminaries}
 
We first introduce the necessary notation and preliminaries.
 Let $\mathbb{N}$ denote the set of positive natural numbers, and let $\mathbb{Z}$ stand for the set of all integers. For any real number $t$, the floor function $\lfloor t\rfloor$ gives the largest integer not larger than $t$, and the ceiling function $\lceil t\rceil$ gives the smallest integer not smaller than $t$. The sum of an infinite series  $\sum_{n\in\mathbb{Z}}c_n$ is defined as  
\[
\sum_{n\in\mathbb{Z}}c_n:=\lim_{N\to\infty}\sum_{|n|\le N}c_n.
\]
For formal definitions and properties of vector and matrix norms, we refer the reader to Chapter 5 in \cite{Horn2012}. We adopt the convention that
\[
\frac{\sin (t)}{t}\Big|_{t=0}:=\lim_{t\to0}\frac{\sin t}{t}=1.
\]
Let $C(\mathbb{R})$, $L^1(\mathbb{R})$, and $L^2(\mathbb{R})$ denote the spaces of continuous functions, Lebesgue integrable functions, and square integrable functions on $\mathbb{R}$, respectively.
The Paley-Wiener space $\mathcal{B}_{\pi}(\mathbb{R})$ \cite[Page 90]{young2001} with bandwidth $\pi$ is defined by 
\[
\mathcal{B}_{\pi}(\mathbb{R}) := \Big\{ f \in C(\mathbb{R}) \cap L^2(\mathbb{R})\mid  {\textrm{supp}}\; \hat{f}\, \subseteq [-\pi,\pi] \Big\},
\]
where the Fourier transform $\hat{f}$ of $f\in L^{1}(\mathbb{R})$ is defined by
\[
\hat{f}(w):=(f){\hat{\,}}:=\frac{1}{\sqrt{2\pi}}\int_{-\infty}^{+\infty} f(x) e^{-iwx}dx,\ w\in{\mathbb R}.
\]

Now, we give a formal definition of the sine-type functions \cite[Lecture 22]{levin1996lectures}. 

\begin{definition}\label{Sinetype}
We say that an entire function $f$ is a sine-type function with bandwidth $\pi$ if it satisfies the following two conditions.
\begin{itemize}
\item[(i)] $\inf_{i\neq j }|\lambda_{i}-\lambda_{j}|\ge\lambda>0$, where $\{\lambda_{i}\}_{i=1}^{\infty}=:\Gamma\subseteq \mathbb{R}$ is the zero point set of $f$;
\item[(ii)] For every $ z\in\mathbb{C}$ satisfying $ \inf_{i\in \mathbb{N} }|z-\lambda_{i}|> \eta>0$, there exists a positive constant $c_{\eta}$  such that 
\[
|f(z)|\ge c_{\eta}e^{\pi|{\textrm Im}(z)|},
\] 
where ${\textrm Im}(z)$ denotes the imaginary part of $z$, and for every $z\in\mathbb{C}$, there exists  a positive constant  $C_{f,1}$ such that 
\[
|f(z)|\le C_{f,1}e^{\pi|{\textrm Im}(z)|}.
\]
\end{itemize}
\end{definition}

Next, we restate Kadec’s $1/4$ theorem as the lemma below, which is an immediate consequence of Theorem 9.8.6 in \cite{Christensen2003} and the elementary fact that the system $\{e^{ikx}\mid k\in\mathbb{Z}\}$ is an orthonormal basis for $L^2([-\pi,\pi])$.

\begin{lemma}\label{Kadec} If $\{\tilde{v}_k\mid  k\in\mathbb{Z}\}\subset \mathbb{R}$ satisfies $|\tilde{v}_k-k|\le L<1/4$ for all $k\in\mathbb{Z}$, then $\{e^{i\tilde{v}_k x}\mid  k\in\mathbb{Z}\}$ forms a Riesz basis for $L^2([-\pi,\pi])$. Specifically, it holds that 
\[
(\cos(\pi L)\!-\!\sin(\pi L))^2\sum_{k\in\mathbb{Z}}|c_k|^2\!\le\! \Big\|\sum_{k\in\mathbb{Z}}c_ke^{i \tilde{v}_k \cdot}\Big\|^2_{L^2([-\pi,\pi])}\!\le\! (2\!-\!\cos(\pi L)+\sin(\pi L))^2\sum_{k\in\mathbb{Z}}|c_k|^2,
\]
for any scales $c_k$.
\end{lemma}

The following lemma is a special case of Theorem 1 in \cite{Zhan1997}, which establishes an upper bound on the spectral norm of Schur products.

\begin{lemma}\label{Lemmarcsingularvalue}
For any matrices $\bC=[c_{ij}:i\in\mathbb{Z}_M,j\in\mathbb{Z}_N]\in\mathbb{C}^{M\times N}$ and $\bD=[d_{ij}:i\in\mathbb{Z}_M,j\in\mathbb{Z}_N]\in\mathbb{C}^{M\times N}$, the Hadamard (or Schur) product of $\bC$ and $\bD$ is denoted by $\bC\circ \bD$. Then,
we have 
\[
\|\bC\circ \bD\|_2\le \min\{{\mathsf{r}}_1(\bC),{\mathsf{c}}_1(\bC)\}\|\bD\|_2,
\]
where ${\mathsf{r}}_1(\bC)=\max_{i\in\mathbb{Z}_M}(\sum_{j\in\mathbb{Z}_N}|c_{ij}|^2)^{1/2}$ and ${\mathsf{c}}_1(\bC)=\max_{j\in\mathbb{Z}_N}(\sum_{i\in\mathbb{Z}_M}|c_{ij}|^2)^{1/2}$ denote the largest row Euclidean norm and largest column Euclidean norm of $\bC$, respectively. 
\end{lemma}

We provide estimates for the largest row Euclidean norm and largest column Euclidean norm of two classes of matrices that will be used later. Their proofs are presented in Appendix B.
\begin{lemma}\label{Cnorm} Let $M\ge2$ be an integer, $0<\epsilon\le\frac18$, $\frac14\le L<\frac12$, and $\Lambda = \{\lambda_k\in[-1/2,M-1/2]\mid  k\in\mathbb{Z}_M\}$  satisfy \eqref{AssumptionLambda}.
Let $T=\{k\in\mathbb{Z}_M\mid  \frac14-\epsilon<|\lambda_k-k|\le L\}$. 
We define the matrix $\widetilde{\bC}_n=[\frac{1}{|j-\lambda_m-nM|}:  m\in T,j\in\mathbb{Z}_M]$ for any $n\in\{0,\pm 1\}$. 
Then, we have 
\[
\min\{{\mathsf{r}}_1(\widetilde{\bC}_0),{\mathsf{c}}_1(\widetilde{\bC}_0)\}<\frac{\pi}{\sin(\pi/8)}\mbox{ and }\min\{{\mathsf{r}}_1(\widetilde{\bC}_{\pm1}),{\mathsf{c}}_1(\widetilde{\bC}_{\pm 1})\}<\frac{\pi}{\sqrt{2}}.
\]
\end{lemma}

\begin{lemma}\label{ExampleCnorm} Let $M\ge2$ be an integer and $\Lambda = \{\lambda_k\in[-1/2,M-1/2]\mid  k\in\mathbb{Z}_M\}$. Let $\Lambda_1=\{v_k\in[-1/4,M-3/4]\mid  k\in\mathbb{Z}_M\}$. For any index set $T\subset \mathbb{Z}_{M}$, define the matrix $\widetilde{\widetilde{\bC}}_n=[\frac{1}{j-\lambda_m-nM}\frac{j-v_m}{j-v_m-nM}(-1)^{nM}:  m\in T, j\in\mathbb{Z}_M ]$ for any $|n|\ge 2$. 
Then, we have for any $|n|\ge2$
\[
\min\{{\mathsf{r}}_1(\widetilde{\widetilde{\bC}}_n),{\mathsf{c}}_1(\widetilde{\widetilde{\bC}}_n)\}\le\frac{1}{\sqrt{2}(|n|-1)^2}.
\]
\end{lemma}

\subsection{Periodic nonuniform interpolation matrices}

 Given a set $\Lambda:=\{\lambda_{k}\mid  k\in\mathbb{Z}_M\}\subset\mathbb{R}$ with
\[
\min_{j,k\in\mathbb{Z}_M,j\neq k}
\operatorname{dist}_M(\lambda_j-\lambda_{k})>0,
\]
we define
\begin{equation}\label{FLambdax}
F_{\Lambda}(x)
:=
\prod_{k\in\mathbb{Z}_M}
\sin\Big(\frac{\pi(x-\lambda_{k})}{M}\Big),
\quad x\in\mathbb{R}.
\end{equation}
Then, by Definition \ref{Sinetype}, $F_{\Lambda}$ is a sine-type function
with bandwidth $\pi$. Let
\begin{equation}\label{taumn}
\tau_{mn}
:=
\lambda_m+nM,
\quad
m\in\mathbb{Z}_M,\ n\in\mathbb{Z}.
\end{equation}
One sees that
$\{\tau_{mn}\mid m\in\mathbb{Z}_M,n\in\mathbb{Z}\}$
is the zero point set of $F_{\Lambda}(x)$ defined as in
\eqref{FLambdax}. Let us define the Lagrange fundamental interpolation
function
\begin{align}\label{psimn}
\psi_{mn}(x)
:=
\frac{F_{\Lambda}(x)}
{F_{\Lambda}^{'}(\tau_{mn})(x-\tau_{mn})}
&=
\frac{
M\prod_{k\in\mathbb{Z}_M}
\sin\frac{\pi(x-\tau_{kn})}{M}
}{
\pi(x-\tau_{mn})
\prod_{k\in\mathbb{Z}_M\setminus\{m\}}
\sin\frac{\pi(\lambda_m-\lambda_k)}{M}
}
\nonumber\\
&=
\frac{
M\prod_{k\in\mathbb{Z}_M}
\sin\frac{\pi(x-\lambda_k)}{M}
}{
(-1)^{nM}\pi(x-\tau_{mn})
\prod_{k\in\mathbb{Z}_M\setminus\{m\}}
\sin\frac{\pi(\lambda_m-\lambda_k)}{M}
},
\end{align}
for any $m\in\mathbb{Z}_M$ and $n\in\mathbb{Z}$, where
$F_{\Lambda}'$ denotes the derivative of $F_{\Lambda}$.

The series
$\sum_{n\in\mathbb{Z}}\sum_{m=0}^{M-1}f(\tau_{mn})\psi_{mn}(x)$
is called the periodic nonuniform sampling series
\cite{Yen1956,ButzerHinsen1989,eldar2002filterbank,strohmer2006,annaby2016bounds,WangWuChen2019}.
The engineering background of periodic nonuniform sampling is multichannel
Analog-to-Digital Conversion \cite{PrendergastLevyHurst2004}. According to
\cite[Theorem 1 in Lecture 23]{levin1996lectures},
$\{\psi_{mn}\mid m\in\mathbb{Z}_M,n\in\mathbb{Z}\}$
forms a Riesz basis of $\mathcal{B}_{\pi}(\mathbb{R})$ and
\begin{equation}\label{Biorth}
\frac{1}{\sqrt{2\pi}}
\int_{-\pi}^{\pi}
\widehat{\psi_{mn}}(x)
\exp(i\tau_{pq}x)\,dx
=
\psi_{mn}(\tau_{pq})
=
\delta_{\tau_{mn},\tau_{pq}},
\quad
m,p\in\mathbb{Z}_M,\ n,q\in\mathbb{Z},
\end{equation}
where $\delta_{x,y}$ is equal to $1$ if $x=y$, and $0$ otherwise.
Moreover,
$\{\exp(i\tau_{mn}x)\mid m\in\mathbb{Z}_M,n\in\mathbb{Z}\}$
forms a Riesz basis of $L^2([-\pi,\pi])$
(see \cite[Theorem 2 in Lecture 23]{levin1996lectures}).

By \eqref{Biorth}, these two sets
\[
\Big\{e^{i\tau_{mn}x}\mid m\in\mathbb{Z}_M,n\in\mathbb{Z}\Big\}
\mbox{ and }
\Big\{
\frac{1}{\sqrt{2\pi}}
\overline{\widehat{\psi_{mn}}}\mid 
m\in\mathbb{Z}_M,n\in\mathbb{Z}
\Big\}
\]
form a pair of biorthogonal bases. Therefore, for any
$j\in\mathbb{Z}_M$, the complex exponential function satisfies
\begin{equation}\label{gs0pre}
\exp(ijy)
=
\sum_{n\in\mathbb{Z}}
\sum_{m\in\mathbb{Z}_M}
\Big\langle
\exp(ij\cdot),
\frac{1}{\sqrt{2\pi}}
\overline{\widehat{\psi_{mn}}}
\Big\rangle_{L^2([-\pi,\pi])}
e^{i\tau_{mn}y}
=
\sum_{n\in\mathbb{Z}}
\sum_{m\in\mathbb{Z}_M}
\psi_{mn}(j)e^{i\tau_{mn}y}.
\end{equation}
The equality \eqref{gs0pre} holds in the $L^2([-\pi,\pi])$ sense. It also holds pointwise for every $y\in(-\pi,\pi)$, as shown in Appendix A.

To proceed, we introduce the modified interpolation function
\begin{equation}\label{psimnstar}
\psi_{mn}^{*}(x)
:=
e^{-inM/10}\psi_{mn}(x),
\quad
m\in\mathbb{Z}_M,\ n\in\mathbb{Z}.
\end{equation}
By \eqref{psimnstar}, we have
$\psi_{mn}(j)=e^{inM/10}\psi_{mn}^{*}(j)$. 
Setting
$y=x-1/10$ in \eqref{gs0pre}, multiplying both sides by
$e^{ij/10}$, and invoking \eqref{taumn}, we obtain
\begin{align}\label{gs0}
\exp(ijx)
&=
\sum_{n\in\mathbb{Z}}
\sum_{m\in\mathbb{Z}_M}
e^{ij/10}e^{inM/10}
\psi_{mn}^{*}(j)
\exp\big(i\tau_{mn}(x-1/10)\big)
\nonumber\\
&=
\sum_{n\in\mathbb{Z}}
\sum_{m\in\mathbb{Z}_M}
e^{i(j-\lambda_m)/10}
\psi_{mn}^{*}(j)
\exp(i\tau_{mn}x).
\end{align}
The equality in \eqref{gs0} holds in the
$L^2([-\pi+1/10,\pi+1/10])$ sense and pointwise for every
$x\in(-\pi+1/10,\pi+1/10)$.

Our goal is to estimate the spectral norm $\|\bA_{\sigma,L}^{-1}\|_2$ of the nonuniform Fourier matrix $\bA_{\sigma,L}^{-1}$, which is defined  in \eqref{Fouriermatrix}. By the following lemma, the spectral norm estimate for $\bA_{\sigma,L}^{-1}$ is reduced to that of periodic nonuniform interpolation matrices.

\begin{lemma}\label{LemmaBsigmaL}
Let $s\le M$, and define an $M$-by-$M$ diagonal matrix
\begin{equation}\label{eqbIs}
\bI_s
=
\operatorname{diag}
(\underbrace{1,1,\dots,1}_{s},
\underbrace{0,0,\dots,0}_{M-s}).
\end{equation}
Moreover, let $\bB_{\sigma,L}$ be the matrix given by
\begin{equation}\label{DefeqBsigmaL}
\bB_{\sigma,L}
:=
\Big[
\sum_{n\in\mathbb{Z}}
\psi_{mn}^{*}(j)
:
m\in\mathbb{Z}_M,\ j\in\mathbb{Z}_M
\Big],
\end{equation}
where $\psi_{mn}^{*}$ is defined as in \eqref{psimnstar},
and let $\bA_{\sigma,L}$ be defined as in \eqref{Fouriermatrix}. Then
\begin{equation*}
\sqrt{M}\|\bI_s\bA_{\sigma,L}^{-1}\|_2
=
\|\bI_s\bB_{\sigma,L}\|_2.
\end{equation*}
\end{lemma}
\begin{proof}
Set
$r_0
:=
\lceil
-\frac{M}{2}-\frac{M}{20\pi}
\rceil$. Then the points
\[
x_k:=-\frac{2\pi(r_0+k)}{M}\in\Big(-\pi+\frac{1}{10},
\pi+\frac{1}{10}\Big)\mbox{ for all }
k\in\mathbb{Z}_M.
\]
Define two $M\times M$ matrices
\[
\widetilde{\bD}
:=
\Big[
\exp\Big(
-\frac{2\pi i(r_0+k)j}{M}
\Big):
k,j\in\mathbb{Z}_M
\Big]
\]
and
\begin{equation}\label{Atildedef}
\widetilde{\bA}_{\sigma,L}
:=
\Big[
\exp\Big(
-\frac{2\pi i(r_0+k)\lambda_m}{M}
\Big):
k,m\in\mathbb{Z}_M
\Big].
\end{equation}
Since the row indices
$r_0,r_0+1,\ldots,r_0+M-1$ are $M$ consecutive integers,
$\widetilde{\bD}/\sqrt{M}$ is unitary.

Applying \eqref{gs0} at $x=x_k$, and noting that $e^{-2\pi in(r_0+k)}=1$ for any $n,k\in\mathbb{Z}_M$,
we have for each $j\in\mathbb{Z}_M$,
\begin{align*}
\exp\Big(
-\frac{2\pi ij(r_0+k)}{M}
\Big)
&=
\sum_{n\in\mathbb{Z}}
\sum_{m\in\mathbb{Z}_M}
e^{i(j-\lambda_m)/10}
\psi_{mn}^{*}(j)
\exp\Big(
-\frac{2\pi i(\lambda_m+nM)(r_0+k)}{M}
\Big)
\\
&=
\sum_{m\in\mathbb{Z}_M}
e^{i(j-\lambda_m)/10}
\Big(
\sum_{n\in\mathbb{Z}}
\psi_{mn}^{*}(j)
\Big)
\exp\Big(
-\frac{2\pi i(r_0+k)\lambda_m}{M}
\Big),
\end{align*}
where the equality holds pointwise for every $k\in\mathbb{Z}_M$.

Define the diagonal unitary matrices
\[
\bP
:=
\operatorname{diag}
\big(
e^{-i\lambda_m/10}:
m\in\mathbb{Z}_M
\big)
\mbox{ and }
\bQ
:=
\operatorname{diag}
\big(
e^{ij/10}:
j\in\mathbb{Z}_M
\big).
\]
Varying the index $j\in\mathbb{Z}_M$, arranging the
above column vectors into a matrix by columns, and using
\eqref{DefeqBsigmaL}, we obtain the fundamental factorization
\begin{equation}\label{shiftmatrixidentity}
\widetilde{\bD}
=
\widetilde{\bA}_{\sigma,L}
\bP
\bB_{\sigma,L}
\bQ.
\end{equation}
Recalling the definitions of $\bA_{\sigma,L}$ in \eqref{Fouriermatrix} and $\widetilde{\bA}_{\sigma,L}$ in \eqref{Atildedef},  a direct comparison of the entries gives
\begin{equation}\label{AtildeA}
\widetilde{\bA}_{\sigma,L}
=
\bA_{\sigma,L}\bE,
\end{equation}
with the unitary diagonal matrix
\[
\bE
:=
\operatorname{diag}
\Big(
\exp\Big(
-\frac{2\pi i r_0\lambda_m}{M}
\Big):
m\in\mathbb{Z}_M
\Big).
\]
Combining
\eqref{shiftmatrixidentity} and \eqref{AtildeA}, we obtain
\begin{equation}\label{Bfactorization}
\bB_{\sigma,L}
=
\bP^{-1}
\bE^{-1}
\bA_{\sigma,L}^{-1}
\widetilde{\bD}
\bQ^{-1}.
\end{equation}
Since the three matrices $\bI_s$ in \eqref{eqbIs}, $\bP$, and $\bE$ are diagonal, they commute pairwise. Hence, multiplying \eqref{Bfactorization} on the left by $\bI_s$ gives
\[
\bI_s \bB_{\sigma,L} = \bP^{-1}\bE^{-1}\bI_s \bA_{\sigma,L}^{-1}\widetilde{\bD}\bQ^{-1}.
\]
Taking the spectral norm and invoking its unitary invariance \cite[Page 320]{Horn2012}, we obtain
\[
\bigl\|\bI_s \bB_{\sigma,L}\bigr\|_2=\|\bI_s \bA_{\sigma,L}^{-1}\widetilde{\bD}\bQ^{-1}\|_2 = \bigl\|\bI_s \bA_{\sigma,L}^{-1}\widetilde{\bD}\bigr\|_2.
\]
Finally, because $\widetilde{\bD}/\sqrt{M}$ is unitary,
\[
\bigl\|\bI_s \bA_{\sigma,L}^{-1}\widetilde{\bD}\bigr\|_2 = \sqrt{M}\,\bigl\|\bI_s \bA_{\sigma,L}^{-1}\bigr\|_2.
\]
Consequently,
\(
\bigl\|\bI_s \bB_{\sigma,L}\bigr\|_2 = \sqrt{M}\,\bigl\|\bI_s \bA_{\sigma,L}^{-1}\bigr\|_2,
\)
which completes the proof. 
\end{proof}

We provide a concrete example demonstrating Lemma \ref{LemmaBsigmaL}.
\begin{example} Let $M=2$, $\lambda_0=0$, and $\lambda_1=2/3$. Then, $r_0:=\lceil -\frac{M}{2}-\frac{M}{20\pi}\rceil=\lceil -1-\frac{1}{10\pi}\rceil=-1$ and associated sampling points $x_0=\pi$ and $x_1=0$.
Both points lie in the interval $(-\pi+1/10,\pi+1/10)$. Choose $\sigma=2/3$ and $L=1/3$. We compute
\[
\bA_{\sigma,L}=
\begin{bmatrix}
1 & 1\\
1 & e^{-2\pi i/3}
\end{bmatrix}
\mbox{ and }
\bA_{\sigma,L}^{-1}=\frac{1}{e^{-2\pi i/3}-1}
\begin{bmatrix}
e^{-2\pi i/3} & -1\\
-1 & 1
\end{bmatrix}.
\]
The singular values of $\bA^{-1}_{\sigma,L}$ are $1$ and $1/\sqrt{3}$.

For $M=2$ and $\Lambda=\{0,2/3\}$, the sine-type function is $F_{\Lambda}(x)=\sin\frac{\pi x}{2}\sin\frac{\pi(x-2/3)}{2}$,
and the fundamental interpolation functions satisfy
$\psi_{0n}(x)=-\frac{4}{\pi\sqrt3}\frac{F_{\Lambda}(x)}{x-2n}$, $\psi_{1n}(x)=\frac{4}{\pi\sqrt3}\frac{F_{\Lambda}(x)}{x-2/3-2n}$.
Recalling that $\psi_{mn}^*(x)=e^{-in/5}\psi_{mn}(x)$, we evaluate the entries of $\bB_{\sigma,L}$ row by row. Let $(\bB_{\sigma,L})_{ij}$ denote the $(i,j)$-th entry of the matrix $\bB_{\sigma,L}$.
For the column $j=0$, only the term $m=n=0$ contributes, giving
$(\bB_{\sigma,L})_{00}=1$ and $(\bB_{\sigma,L})_{10}=0$.

For the remaining two entries, we invoke the Fourier series for $e^{i\beta x}$ on $[-\pi,\pi]$:
\begin{equation}\label{Fourierexpansion}
\sum_{k\in\mathbb{Z}}\frac{e^{ikx}}{k+\beta}
=
\frac{\pi e^{i\beta (\pi-x)}}{\sin(\pi \beta)}\mbox{ for any }0<x<2\pi\mbox{ and any }\beta\notin\mathbb{Z},
\end{equation}
where the series is understood in the symmetric sense (see \cite[Chapter~2, Example~3]{SteinShakarchiFourier}).

For the column \(j=1\), we evaluate at \(x=1\). Noting that \(F_{\Lambda}(1)=\sin(\pi/6)=1/2\), we obtain $\psi_{0n}(1)=\frac{2}{\pi\sqrt3}\frac1{2n-1}$ and 
\[
(\bB_{\sigma,L})_{01}
\!=\!\sum_{n\in\mathbb Z}e^{-in/5}\psi_{0n}(1)
\!=\!\frac{1}{\pi\sqrt3}\sum_{n\in\mathbb Z}\frac{e^{in(2\pi-1/5)}}{n-1/2}\!=\!\frac{1}{\pi\sqrt3}\frac{\pi e^{i(-\frac12)(\pi-2\pi+\frac15)} }{\sin(-\pi/2)}
\!=\!-\frac{i}{\sqrt3}e^{-\frac{i}{10}},
\]
where we have used  \eqref{Fourierexpansion} with \(x=2\pi-1/5\) and \(\beta=-1/2\) in the third equality.
Similarly, we obtain $\psi_{1n}(1)=\frac{2}{\pi\sqrt3}\frac1{1/3-2n}$ and 
\[
(\bB_{\sigma,L})_{11}
\!=\!
-\frac{2}{\pi\sqrt3}\sum_{n\in\mathbb Z}\frac{e^{-in/5}}{2n-1/3}
\!=\!
-\frac{1}{\pi\sqrt3}\sum_{n\in\mathbb Z}\frac{e^{in(2\pi-\frac15)}}{n-1/6}\!=\!-\frac{1}{\pi\sqrt3}\frac{\pi e^{i(-\frac16)(\pi-2\pi+\frac15)}}{\sin(-\pi/6)}\!=\!\frac{2}{\sqrt{3}}e^{i(\frac{\pi}{6}-\frac{1}{30})},
\]
where we have used \eqref{Fourierexpansion} with \(x=2\pi-1/5\) and \(\beta=-1/6\) in the third equality.

Consequently,
\[
\bB_{\sigma,L}
=
\begin{bmatrix}
1 & -\frac{i}{\sqrt3}e^{-\frac{i}{10}}\\
0 & \frac{2}{\sqrt{3}}e^{i(\frac{\pi}{6}-\frac{1}{30})}
\end{bmatrix}.
\]
A direct computation yields that the singular values of  $\bB_{\sigma,L}$ are $\sqrt{2}$ and $\sqrt{2}/\sqrt{3}$. Therefore, 
$\sqrt{2}\|\bI_s\bA_{\sigma,L}^{-1}\|_2
=\|\bI_s\bB_{\sigma,L}\|_2$ for any $s\in\{1,2\}$. Furthermore, 
\begin{align*}
&\bP^{-1}
\bE^{-1}
\bA_{\sigma,L}^{-1}
\widetilde{\bD}
\bQ^{-1}\\
&={\textrm diag}(1,\!e^{\frac{i}{15}}){\textrm diag}(1,\!e^{-\frac{2\pi i}{3}})
\frac{1}{e^{-\!\frac{2\pi i}{3}}\!-\!1}
\begin{bmatrix}
e^{-\!\frac{2\pi i}{3}} & -1\\
-1 & 1
\end{bmatrix}\begin{bmatrix}1&-1\\ 1&1\end{bmatrix}{\textrm diag}(1,e^{-\!\frac{i}{10}})=\bB_{\sigma,L}
\end{align*}
verifies \eqref{Bfactorization}.
\end{example}

\begin{remark}\label{RemarkKadec}
The discrete version of Kadec's 1/4 theorem, presented in \cite[Theorem 2.2]{asipchuk2025concerning} and \cite[Corollary 2.1]{yu2023on}, asserts that if $\{\lambda_k\mid  k\in\mathbb{Z}_M\}\subset \mathbb{R}$ satisfies $|\lambda_k-k|\le L<1/4$ for all $k\in\mathbb{Z}_M$, then for any vector $\ba\in\mathbb{C}^{M}$ with $\|\ba\|_2=1$, we have
\[
\cos(\pi L)-\sin(\pi L)\le \frac{\|\bA_{\sigma,L}\ba\|_2}{\sqrt{M}}\le 2-\cos(\pi L)+\sin(\pi L). 
\]
 In addition, the earlier work \cite{MarzoSeip2009} also establishes an upper bound for $\|\bA_{\sigma,L}^{-1}\|_2$ under the condition $L\le 1/4-\epsilon$. However, it does not provide an explicit expression for the constant. As $L\le \frac14-\epsilon$ with $0<\epsilon\le 1/8$, the inequality $\cos(\pi L)-\sin(\pi L)>(1-4L)\ge 4\epsilon$ is always satisfied.  This immediately implies the estimate
\begin{equation}\label{1/4}    
\|\bB_{\sigma,L}\|_2=\sqrt{M}\|\bA_{\sigma,L}^{-1}\|_2\le\frac{1}{\cos(\pi L)-\sin(\pi L)}\le \frac{1}{4\epsilon}
\quad \text{for any }  L\le 1/4-\epsilon,
\end{equation}
where the matrix $\bB_{\sigma,L}$ is defined as in \eqref{DefeqBsigmaL}.
\end{remark}

Let $\Lambda_1=\{v_k\mid  k\in\mathbb{Z}_M\}\subset\mathbb{R}$ be another set of points, chosen depending on $\Lambda$ so as to enable the application of Kadec's $1/4$ theorem. We introduce a suitable decomposition of $\psi_{mn}^{*}(x)$. Define 
\begin{equation}\label{DefV}
V(x):=\frac{F_{\Lambda}(x)}{F_{\Lambda_1}(x)}=\prod_{k\in\mathbb{Z}_M}\frac{\sin\frac{\pi(x-\lambda_{k})}{M}}{\sin\frac{\pi(x-v_{k})}{M}}.
\end{equation}
Whenever $F_{\Lambda}$ and  $F_{\Lambda_1}$ in \eqref{DefV} have common zeros, the quotient $V(x)=F_{\Lambda}(x)/F_{\Lambda_1}(x)$ is understood after cancellation of the common factors.
Let
\[
R_{mn}(x):=\frac{x-v_m-nM}{x-\lambda_m-nM}V(x),
\]
and
\begin{equation}\label{phimn}
\phi_{mn}(x):=\frac{F_{\Lambda_1}(x)e^{-inM/10}}{F_{\Lambda_1}'(v_m+nM)(x-v_m-nM)},
\end{equation}
where $m\in\mathbb{Z}_M$, $n\in\mathbb{Z}$, $x\in\mathbb{R}$, $F_{\Lambda}$ and $F_{\Lambda_1}$ are defined as in \eqref{FLambdax}. Notice that $\phi_{mn}$ coincides with $\psi_{mn}^{*}$ defined as in \eqref{psimn} when $\Lambda_1=\Lambda$. It follows that
\begin{equation}\label{VRQ1}
R_{mn}(x)\phi_{mn}(x)=\frac{F_{\Lambda}(x)e^{-inM/10}}{F_{\Lambda_1}'(v_m\!+\!nM)(x\!-\!\lambda_m\!-\!nM)} 
\end{equation}
and
\begin{equation}\label{VRQ2}
R_m(\lambda_m+nM)\phi_{m0}(\lambda_m+nM)=\frac{F'_{\Lambda}(\lambda_m+nM)e^{-inM/10}}{F_{\Lambda_1}'(v_m+nM)},
\end{equation}
where in the second identity we have used the fact
\[
\frac{F_{\Lambda}(x)}{x\!-\lambda_m\!-nM}\Big|_{x=\lambda_m+nM}\!\!=\!\!\lim_{x\to\lambda_m+nM}\frac{F_{\Lambda}(x)\!-\!0}{x\!-\!\lambda_m-nM}\!\!=\!\!\lim_{x\to\lambda_m+nM}\frac{F_{\Lambda}(x)\!\!-\!\!F_{\Lambda}(\lambda_m+nM)}{x\!-\!\lambda_m-nM}\!\!=\!\!F'_{\Lambda}(\lambda_m+nM).
\]
By \eqref{VRQ1} and \eqref{VRQ2}, we can express $\psi_{mn}^{*}$ in \eqref{psimn} in terms of $\phi_{mn}$ in \eqref{phimn} as follows
\[
\psi_{mn}^{*}(x)=\frac{R_{mn}(x)\phi_{mn} (x)}{R_{mn}(\lambda_m+nM)\phi_{mn}(\lambda_m+nM)}
=\frac{V(x)\frac{x-v_m-nM}{x-\lambda_m-nM}\phi_{mn}(x)}{F'_{\Lambda}(\lambda_m+nM)/F_{\Lambda_1}'(v_m+nM)},\ x\in\mathbb{R}.
\]
Recall that $F_{\Lambda}$ and $F_{\Lambda_1}$ are defined as in \eqref{FLambdax}. It follows that
\[
\frac{F'_{\Lambda_1}(x+nM)}{F'_{\Lambda}(x+nM)}=\frac{F'_{\Lambda_1}(x )}{F'_{\Lambda}(x )}\mbox{ for any }n\in\mathbb{Z}.
\]
As a result, we obtain
\begin{equation}\label{psimnrewrite}
\psi_{mn}^{*}(x)
=V(x)\frac{ F_{\Lambda_1}'(v_m)}{F'_{\Lambda}(\lambda_m)}\frac{x-v_m-nM}{x-\lambda_m-nM}\phi_{mn}(x),\ x\in\mathbb{R},\ m\in\mathbb{Z}_M, n\in\mathbb{Z}.
\end{equation}

Using  Kadec's 1/4 theorem, we also have the following lemma.
\begin{lemma}\label{27} Let \(0<\epsilon<1/4\). Assume that $\Lambda_1=\{v_{k}\mid  k\in\mathbb{Z}_M\}$ satisfies $|v_{k}-k|\le 1/4-\epsilon$ for all $k\in\mathbb{Z}_M$. Then, we have for any $n\in\mathbb{Z}$ 
\[  
\| [\phi_{mn}(j): m,j\in\mathbb{Z}_M]\|_2\le \frac{1}{4\epsilon}, 
\]
where the entry $\phi_{mn}(j)$ is defined as in \eqref{phimn}.
\end{lemma}
\begin{proof}
For any vector \(\bb=(b_j: j\in\mathbb{Z}_M)\) with $\|\bb\|_2=1$, we let 
\[
f(x)=\sum_{j\in\mathbb{Z}_M}b_{j}\exp(ijx),\ x\in[-\pi,\pi].
\]
Then \(\|f\|_{L^{2}([-\pi,\pi])}=\|\bb\|_2=1\).  
By \eqref{Biorth} and \eqref{phimn}, the two systems 
\[
\big\{e^{inM/10}e^{i(v_m+nM)x}\mid  m\in\mathbb{Z}_M, n\in\mathbb{Z}\big\} 
\mbox{ and }
\Big\{\frac{1}{\sqrt{2\pi}}\overline{\widehat{\phi_{mn}}}\mid m\in\mathbb{Z}_M, n\in\mathbb{Z}\Big\}
\]
also form a pair of biorthogonal bases for $L^2([-\pi,\pi])$ as
\begin{align*}
&\frac{1}{\sqrt{2\pi}}
\int_{-\pi}^{\pi}
\widehat{\phi_{mn}}(x)
e^{iqM/10}e^{i(v_p+qM)x}\,dx\\
&=\frac{1}{\sqrt{2\pi}}
\int_{-\pi}^{\pi}\Big(\frac{F_{\Lambda_1}(\cdot)}{F_{\Lambda_1}'(v_m+nM)(\cdot-v_m-nM)}\Big)^{\hat{\,}}(x) e^{-inM/10}\cdot 
e^{iqM/10}e^{i(v_p+qM)x}\,dx\\
&=e^{i(q-n)M/10}\delta_{mp}\delta_{nq}=\delta_{mp}\delta_{nq}\, 
\mbox{ for any }m,p\in\mathbb{Z}_M,\ n,q\in\mathbb{Z}.
\end{align*}

Combining this with \eqref{gs0}, we have
\[
f(x)=\sum_{n\in\mathbb{Z}}\sum_{m\in \mathbb{Z}_{M}}  (c_{m,n}e^{inM/10})\exp(i (v_{m}+nM)x),
\]
with
\[
c_{m,n}=\frac{1}{\sqrt{2\pi}}\int_{-\pi}^{\pi} f(x)\widehat{\phi_{mn}}(x)dx
=\sum_{j\in\mathbb{Z}_M}b_j\frac{1}{\sqrt{2\pi}}\int_{-\pi}^{\pi} e^{ijx}\widehat{\phi_{mn}}(x)dx
=\sum_{j\in\mathbb{Z}_M}b_{j}\phi_{mn}(j).
\]
Using Lemma \ref{Kadec} with $L:=1/4-\epsilon$ and $\tilde{v}_{m+nM}=v_m+nM$, and invoking the inequality $\cos(\pi L)-\sin(\pi L)>(1-4L)=4\epsilon$, we have 
\[
\sum_{n\in \mathbb{Z}}\sum_{m\in \mathbb{Z}_{M}}\big|c_{m,n}e^{
\frac{inM}{10}}\big|^2\!=\!\sum_{n\in \mathbb{Z}}\sum_{m\in \mathbb{Z}_{M}}|c_{m,n}|^2
\!\le\! \frac{\|f\|_{L^{2}([-\pi,\pi])}^{2}}{(\cos(\pi L)-\sin(\pi L))^2}
\!=\! \frac{1}{(\cos(\pi L)-\sin(\pi L))^2}
\le \frac{1}{(4\epsilon)^2}.
\]
Then, 
\[
\big\|[\phi_{mn}(j): m,j\in\mathbb{Z}_M]\bb\big\|^2_2=\sum_{m\in \mathbb{Z}_{M}}\Big|\sum_{j\in\mathbb{Z}_M}b_{j}\phi_{mn}(j)\Big|^2 =\sum_{m\in \mathbb{Z}_{M} }|c_{m,n}|^2\le \frac{1}{(4\epsilon)^2},
\]
which completes the proof.
\end{proof}

The following lemma reduces the estimation of  \(\|\bB_{\sigma,L}\|_2\) to estimating \(V(x)\), \(V_{m}(x)\) and \( F_{\Lambda_1}'(v_m)/F'_{\Lambda}(\lambda_m).\)

\begin{lemma}\label{LLLL27}  Let
\(\Lambda=\{\lambda_k\mid k\in\mathbb Z_M\}\subset[-1/2,M-1/2)\),
\(\Lambda_1=\{v_k\mid k\in\mathbb Z_M\}\subset[-1/2,M-1/2)
\)
be two sets of points that are pairwise distinct modulo $M$,  Let $1\le s\le M$, assume that $ \max_{m\in \mathbb{Z}_{s}}|\lambda_m-v_m|\le M/2$. For any $m\in\mathbb Z_s$, set
\[
V_m(x):=
\prod_{k\in\mathbb Z_M\setminus\{m\}}
\frac{\sin\frac{\pi(x-\lambda_k)}{M}}
     {\sin\frac{\pi(x-v_k)}{M}},
\] 
where common factors are cancelled. Let $T:=\{k\in\mathbb{Z}_{M}\mid  v_k\neq \lambda_k\}$. Assume that 
\(\
\min_{ m\in T\cap\mathbb{Z}_{s}}\operatorname{dist}(v_m,\mathbb Z)\gtrsim 1\) for all \( m\in\mathbb Z_s\).
For any $n\in\mathbb{Z}$, define 
\[
\bC_n\!:=\!\Big[\frac{\lambda_m-v_m}{j-\lambda_m-nM}\frac{j-v_m}{j-v_m-nM}(-1)^{nM}e^{-inM/10}:   m\in \mathbb{Z}_s,j\in\mathbb{Z}_M\Big],
\] 
\[
\bD_n\!:=\![\phi_{mn}(j):  m\in \mathbb{Z}_s,j\in\mathbb{Z}_M], \quad \bD^{*}:=\Big[\sum_{n\in\mathbb{Z}}\phi_{mn}(j):  m\in \mathbb{Z}_s,j\in\mathbb{Z}_M\Big].
\]
For $|n|<2$, define
\[
\bH_n:=
[V_m(j)q_{mj}^{(n)}: {m\in\mathbb Z_s,\,j\in\mathbb Z_M}],
\] 
with
\[
q_{mj}^{(n)}
:=\left\{\begin{array}{ll}
\frac{\lambda_m-v_m}{j-\lambda_m-nM}
\frac{\sin\frac{\pi(j-\lambda_m)}{M}}
     {\sin\frac{\pi(j-v_m)}{M}}, & \text{ if }m\in T,\\
0, &\mbox{ if } m\notin T.
\end{array}\right.
\] 
Then 
\begin{align*}
\| \bI_s\bB_{\sigma,L}\|_2&\lesssim \! \sup_{m\in \mathbb{Z}_s}\Big| 
 \frac{F'_{\Lambda_1}(v_m)}{F'_{\Lambda}(\lambda_m)}\Big|\Big(\sup_{j\in \mathbb{Z}_{M}}|V(j)|\Big(\|\bD^{*}\|_2\! +\Big\|\sum_{|n|\ge 2}\bC_n\circ \bD_0\Big\|_2\Big)\\&+\sup_{m\in \mathbb{Z}_{s}}\operatorname{dist}_M(\lambda_m,v_m)\sup_{j,m\in \mathbb{Z}_{M}}|V_{m}(j)|\sup_{|n|<2}\|\bD_n\|_2\Big).\end{align*}
\end{lemma}
\begin{proof} Let
\(
\bP_s:=
\begin{bmatrix}\bE_s,{\textbf 0}\end{bmatrix}\in\mathbb C^{s\times M},
\) where $\bE_s$ is an $s$-by‑$s$ idenity matrix.
Then
\(
\|\bP_s\bB_{\sigma,L}\|_2
=
\|\bI_s\bB_{\sigma,L}\|_2,
\)
where the $M$-by-$M$ diagonal matrix $\bI_s$ is defined as in \eqref{eqbIs}.
For any vector $\ba=(a_j: j\in\mathbb{Z}_M)$ with $\|\ba\|_2=1$.  By Lemma \ref{LemmaBsigmaL} and \eqref{psimnrewrite}, we have 
\begin{align*}
\| \bP_s\bB_{\sigma,L}\ba\|_2
&=\Big(\sum_{m\in \mathbb{Z}_{s}}\Big|\sum_{j\in \mathbb{Z}_{M}}a_{j}\sum_{n\in\mathbb{Z}}\psi_{mn}^{*}(j)\Big|^{2}\Big)^{1/2} \\
&\le  \Big( \sup_{m\in \mathbb{Z}_{s}}\Big| \frac{F'_{\Lambda_1}(v_m)}{F'_{\Lambda}(\lambda_m)}\Big|\Big) \Big(\sum_{m\in \mathbb{Z}_{s}}\Big|\sum_{j\in \mathbb{Z}_{M}}a_{j}V(j)\sum_{n\in\mathbb{Z}}\frac{j-v_m-nM}{j-\lambda_m-nM}\phi_{mn}(j)\Big|^{2} \Big)^{1/2} \\
&\le \Big( \sup_{m\in \mathbb{Z}_{s}}\Big| 
 \frac{F'_{\Lambda_1}(v_m)}{F'_{\Lambda}(\lambda_m)}\Big|\Big)  \Big\|\Big[\sum_{n\in\mathbb{Z}}V(j)\frac{j-v_m-nM}{j-\lambda_m-nM}\phi_{mn}(j): m\in \mathbb{Z}_s,j\in\mathbb{Z}_M\Big]\Big\|_2,
\end{align*}
where in the second inequality we have used the fact that $\|\bA \ba\|_2\le \|\ba\|_2\|\bA\|_2=\|\bA\|_2$ for any $\bA\in\mathbb{R}^{s\times M}$ \cite[Theorem 5.6.2 (b)]{Horn2012}. 
Observe that
\[
\frac{j-v_m-nM}{j-\lambda_m-nM}=1+\frac{\lambda_m-v_m}{j-\lambda_m-nM}\mbox{ for all } n\in\mathbb{Z},\ m\in \mathbb{Z}_s,j\in\mathbb{Z}_M.
\]
By \eqref{eqspectral} and Lemma \ref{Lemmarcsingularvalue},  
\begin{align*}
&\|\bP_s\bB_{\sigma,L}\|_2\\
&\le \sup_{m\in \mathbb{Z}_s}\Big| 
 \frac{F'_{\Lambda_1}(v_m)}{F'_{\Lambda}(\lambda_m)}\Big|\Big(\Big\|\Big[V(j)\sum_{n\in\mathbb{Z}}\phi_{mn}(j): m\in \mathbb{Z}_s,j\in\mathbb{Z}_M\Big]\Big\|_2\\
 &\quad+\Big\|\Big[\sum_{n\in\mathbb{Z}}V(j)\frac{\lambda_m-v_m}{j\!-\!\lambda_m\!-\!nM}\phi_{mn}(j): m\in \mathbb{Z}_s,j\in\mathbb{Z}_M\Big]\Big\|_2\Big)\\
&\le \sup_{m\in \mathbb{Z}_s}\Big| 
 \frac{F'_{\Lambda_1}(v_m)}{F'_{\Lambda}(\lambda_m)}\Big|\Big(\sup_{j\in \mathbb{Z}_{M}}|V(j)|\|\bD^{*}\|_2\!+\!\Big\|\Big[\sum_{n\in\mathbb{Z}}V(j)\frac{\lambda_m-v_m}{j\!-\!\lambda_m\!-\!nM}\phi_{mn}(j): m\in \mathbb{Z}_s,j\in\mathbb{Z}_M\Big]\Big\|_2\Big). 
\end{align*}
By the definition of $\phi_{mn}$ in \eqref{phimn},  we have for any $n\in\mathbb{Z}$
\[
\phi_{mn}(j)
=\frac{(-1)^{nM}e^{-\frac{inM}{10}}M\prod_{k\in\mathbb{Z}_M}\sin\frac{\pi(j-v_k)}{M}}{\pi(j-v_m-nM)\prod_{k\in\mathbb{Z}_M\setminus\{m\}}\sin\frac{\pi(v_m-v_k)}{M}}
=\frac{j-v_m}{j-v_m-nM}(-1)^{nM}e^{-\frac{inM}{10}}\phi_{m0}(j).
\]
By the definitions of the set $T$ and the entry $q^{(n)}_{mj}$, we compute for all $|n|<2$
\[
\sqrt{\sum_{j\in \mathbb{Z}_{M}} \Big|V_{m}(j)q_{mj}^{(n)}\Big|^2}=0 \mbox{ for any } m\notin T
\]
and
\begin{align}\label{qmT}
\sqrt{\sum_{j\in \mathbb{Z}_{M}} \Big|V_{m}(j)q_{mj}^{(n)}\Big|^2}
&\le \sup_{j\in \mathbb{Z}_{M}}|V_{m}(j)| \sqrt{\sum_{j\in \mathbb{Z}_{M}} \Big|q_{mj}^{(n)}\Big|^2}\nonumber \\
&\le \sup_{j\in \mathbb{Z}_{M}}|V_{m}(j)| \lambda_m\!-\!v_m|\sqrt{\sum_{j\in \mathbb{Z}_{M}} \Big|\frac{1}{j-\lambda_m-nM}
\frac{\sin\frac{\pi(j-\lambda_m)}{M}}
     {\sin\frac{\pi(j-v_m)}{M}}\Big|^2}\nonumber \\
&\le \sup_{j\in \mathbb{Z}_{M}}|V_{m}(j)| \lambda_m\!-\!v_m|\sqrt{\sum_{j\in \mathbb{Z}_{M}} \Big|\frac{\pi}{2}\frac{1}{\dist(j-v_m,M\mathbb{Z})}\Big|^2}\mbox{ for any }m\in T,
\end{align}
where in the third inequality, we have used $|\sin \pi(j-\lambda_m)/M |=|\sin \pi(j-\lambda_m-nM)/M |\le |j-\lambda_m-nM|\pi/M$ and 
\[
\Big|\sin\frac{\pi(j-v_m)}{M}\Big|\ge 2 \dist\Big(\frac{j-v_m}{M},\mathbb{Z}\Big)=\frac{2}{M}\dist(j-v_m,M\mathbb{Z})
\]
as  \(\
\min_{ m\in T\cap\mathbb{Z}_{s}}\operatorname{dist}(v_m,\mathbb Z)\gtrsim 1\) for all \( m\in\mathbb Z_s\).
Note that the map \( j \mapsto j - v_{\lambda} \pmod{M} \) runs over all residue classes modulo \( M \) and $\mathbb{Z}=\cup_{j\in\mathbb{Z}_m}\{l\in\mathbb{Z}\mid  l=j \pmod{M}\}$. Hence
\begin{equation}\label{qmTbound}
\sum_{j\in\mathbb{Z}_M} \frac{1}{\operatorname{dist}(j - v_{m}, M\mathbb{Z})^2}
\le \sum_{\ell \in \mathbb{Z}} \frac{1}{\operatorname{dist}(v_m, \ell)^2}
=
\sum_{\ell \in \mathbb{Z}} \frac{1}{|v_m - \ell|^2}\lesssim 1,
\end{equation}
where we have used the assumption \(\
\min_{ m\in T\cap\mathbb{Z}_{s}}\operatorname{dist}(v_m,\mathbb Z)\gtrsim 1\) for all \( m\in\mathbb Z_s\). Combining \eqref{qmT} and \eqref{qmTbound}, we have for all $|n|<2$ 
\[
\sqrt{\sum_{j\in \mathbb{Z}_{M}} \Big|V_{m}(j)q_{mj}^{(n)}\Big|^2}\lesssim \sup_{j\in \mathbb{Z}_{M}}|V_{m}(j)|\operatorname{dist}_M(\lambda_m,v_m)\mbox{ for any }m\in T.
\]
As a result, we have for all $|n|<2$ 
\begin{equation}\label{qbound}
\max_{m\in\mathbb{Z}_s}\sqrt{\sum_{j\in \mathbb{Z}_{M}} \Big|V_{m}(j)q_{mj}^{(n)}\Big|^2}\lesssim \sup_{j\in \mathbb{Z}_{M},m\in \mathbb{Z}_s}|V_{m}(j)| \sup_{m\in \mathbb{Z}_{s}}\operatorname{dist}_M(\lambda_m,v_m).
\end{equation}

According to the definitions $\bC_n$, $\bD_n$, and $\bH_n$, it follows that
\begin{align*}
\|\bP_s\bB_{\sigma,L}\|_2
&\le \sup_{m\in \mathbb{Z}_s}\Big| 
 \frac{F'_{\Lambda_1}(v_m)}{F'_{\Lambda}(\lambda_m)}\Big|\Big(\sup_{j\in \mathbb{Z}_{M}}|V(j)|\Big(\|\bD^{*}\|_2\! +\Big\|\sum_{|n|\ge 2}\bC_n\circ \bD_0\Big\|_2\Big)+\sum_{|n|<2}\Big\|\bH_n\circ\bD_n\Big\|_2\Big)\\
 &\lesssim \sup_{m\in \mathbb{Z}_s}\Big| 
 \frac{F'_{\Lambda_1}(v_m)}{F'_{\Lambda}(\lambda_m)}\Big|\Big(\sup_{j\in \mathbb{Z}_{M}}|V(j)|\Big(\|\bD^{*}\|_2\! +\Big\|\sum_{|n|\ge 2}\bC_n\circ \bD_0\Big\|_2\Big)\\
 &\quad+\sup_{m\in \mathbb{Z}_{s}}\operatorname{dist}_M(\lambda_m,v_m)\sup_{j,m\in \mathbb{Z}_{M}}|V_{m}(j)|\sup_{|n|<2}\|\bD_n\|_2\Big),
\end{align*}
where the last inequality follows from Lemma \ref{Lemmarcsingularvalue} and \eqref{qbound}.
The proof is therefore complete.
\end{proof}

The following lemma provides, for any given frequency set $\Lambda$, a careful selection of the corresponding set $\Lambda_1$ in Lemma \ref{LLLL27}.
  
\begin{lemma}\label{lem11}  Let $\Lambda = \{\lambda_k\subset[-1/2,M-1/2]\mid  k\in\mathbb{Z}_M\}$  satisfy the assumption \eqref{AssumptionLambda} and  $0<\epsilon\le \frac{1}{8}$ and $0<  L<1/2$, then there exists $\Lambda_1=\{v_k\subset[-1/4,M-3/4]\mid  k\in\mathbb{Z}_M\}$  such that 
\begin{equation}\label{tiaojian}
\max_{k\in \mathbb{Z}_{M}}\operatorname{dist}_M(v_k,k)\le\frac14\!- \epsilon , \max_{k\in \mathbb{Z}_{M}}\operatorname{dist}_M(\lambda_k,v_k)\le L+\epsilon\!-\!\frac14, \min_{k\in \mathbb{Z}_{M},\lambda_{k}\neq v_{k}}\operatorname{dist}_M(v_k,k)\ge \frac18.
\end{equation}
\end{lemma}
\begin{proof} For each $k\in\mathbb{Z}_M$, we consider two separate cases.

Case 1: If  $|\lambda_k - k| \le \frac{1}{4} - \epsilon$, then we assign $v_k := \lambda_k$. Let $\Lambda_0$ denote the set consisting of all such points $v_k$.
If $\Lambda_0$ is nonempty, then by \eqref{AssumptionLambda},  we have 
$\min_{\lambda_k \notin\Lambda_0} \operatorname{dist}_M(\lambda_k,  \Lambda_{0})\ge \sigma\ge \frac{\sigma}{6}$, $|\lambda_k-v_k|=0$, and 
\[
|v_k-k|=|\lambda_k-k|\le \frac14-\epsilon .
\]

Case 2: If $|\lambda_k-k|> \frac{1}{4}-\epsilon$ and then choose 
\[
v_k:=\left\{\begin{array}{ll}
k + \frac{1}{4} -  \epsilon , &\mbox{ if } \lambda_k>k+\frac{1}{4}-\epsilon,\\
k - \frac{1}{4} +  \epsilon , &\mbox{ if }\lambda_k<k-\frac{1}{4}+\epsilon.
\end{array}
\right.
\]
It follows that $|v_k-k|=\frac{1}{4}-  \epsilon \ge \frac{1}{8}$.
By \eqref{AssumptionLambda}, the definition of $v_k$ and $0\le L<1/2$,  we have
\[
|\lambda_k-v_k| \le  L-\frac{1}{4} +  \epsilon.
\]
The set consisting of all such points $v_k$ is denoted by $\Lambda_{00}$.
Setting $\Lambda_1:=\Lambda_{0}\cup\Lambda_{00}$ completes the proof.
\end{proof}

\begin{corollary}
Let $0\le L<1/2$ and \(0 <\epsilon\le 1/8\).
If $\Lambda$ and $\Lambda_1$ satisfy the conditions of Lemma \ref{lem11}, then 
\begin{equation}\label{eqBsigmaLFinal}
\|\bB_{\sigma,L}\|_2\lesssim  \frac{1}{\epsilon}\Bigg(\sup_{m,j\in \mathbb{Z}_M}\Big(| V(j)|+|V_{m}(j)|\Big)
 \Bigg) \sup_{m\in \mathbb{Z}_M}\Bigg|\frac{F'_{\Lambda_1}(v_m)}{F'_{\Lambda}(\lambda_m)}\Bigg|.
\end{equation}
\end{corollary}
\begin{proof} Using Lemma \ref{27}, $\|\bD_n\|_2\le1/(4\epsilon)$ for all $n\in\mathbb{Z}$.
By the definition of $\bC_n$, we have
\[
(\bC_n)_{mj}=(\lambda_m-v_m)(\widetilde{\widetilde{\bC}}_n)_{mj}, |n|\ge2\mbox{ and }(\bE_n)_{mj}=(\lambda_m-v_m)(\widetilde{\bC}_n)_{mj},n=0,-1,1,
\]
for any $m\in T=\{k\in\mathbb{Z}_M\mid \lambda_k\neq v_k\}$ and any $j\in\mathbb{Z}_M$, where the matrices $\widetilde{\widetilde{\bC}}_n$ and $\widetilde{\bC}_n$ are defined as in  Lemmas \ref{ExampleCnorm} and \ref{Cnorm}, respectively.
By Lemma \ref{Lemmarcsingularvalue},  we obtain
\begin{align}\label{sumeq2}
&\sum_{|n|\ge 2}\Big\|\bC_n\circ \bD_0\Big\|_2
\le \sum_{|n|\ge 2}\min\{{\mathsf{r}}_1(\bC_n),{\mathsf{c}}_1(\bC_n)\}\| \bD_0\|_2\\
&\le\frac{1}{4\epsilon} \Big(\sum_{|n|\ge 2}\min\{{\mathsf{r}}_1(\bC_n),{\mathsf{c}}_1(\bC_n)\}\Big)\nonumber\\
&\le\frac{1}{4\epsilon}\frac{3}{8} \Big(\sum_{|n|\ge 2}\min\{{\mathsf{r}}_1(\widetilde{\widetilde{\bC}}_n),{\mathsf{c}}_1(\widetilde{\widetilde{\bC}}_n)\}\Big)\nonumber\lesssim \frac{1}{\epsilon},
\end{align}
where the third inequality follows from the bound $|\lambda_m-v_m|\le L+\epsilon-\frac14\le \frac38$ for any $m\in \mathbb{Z}_{M}$ and the fourth inequality is a direct consequence of Lemmas \ref{Cnorm} and \ref{ExampleCnorm}.
‌Similarly, by Lemma \ref{27}, we have
\[
   \sup_{j,m\in \mathbb{Z}_{M}}|V_{m}(j)|\sup_{|n|<2}\|\bD_n\|_2\lesssim \frac{1}{\epsilon}\sup_{m,j\in \mathbb{Z}_M}\big|V_{m}(j)\big|.
\]
Combining Lemma \ref{LLLL27} and \eqref{sumeq2}, we have
\[
\|\bB_{\sigma,L}\|_2\lesssim  \frac{1}{\epsilon}\Big(\sup_{m,j\in \mathbb{Z}_M}\Big(| V(j)|+|V_{m}(j)|\Big)
 \Big) \sup_{m\in \mathbb{Z}_M}\Bigg|\frac{F'_{\Lambda_1}(v_m)}{F'_{\Lambda}(\lambda_m)}\Bigg|.
\]
It follows from \eqref{1/4} in Remark \ref{RemarkKadec} and \eqref{phimn} that $\|\bD^*\|_2\le1/(4\epsilon)$. 
which yields \eqref{eqBsigmaLFinal}.
\end{proof}

\section{Proofs of Theorem \ref{Theorem} and Corollary \ref{coro}}\label{Section3}

By \eqref{eqBsigmaLFinal}, it remains to estimate $V(j)$, $V_m(j)$, and $F'_{\Lambda_1}(v_m)/F'_{\Lambda}(\lambda_m)$ in \eqref{eqBsigmaLFinal}. To do this, we need the following two simple lemmas.

\begin{lemma}\label{L23} Let $M\ge2$ be an integer and $\mathcal{T}\subset \mathbb{Z}_{M}$. If there exist $0<\epsilon \le\frac18$, $\frac14\le p\le 1/2$, $\sigma'>0$, and $q>0$ such that $\Gamma_1=\{\gamma_k\mid  k\in \mathcal{T}\}\subset\mathbb{R}$ and $\Gamma_2=\{u_k\mid  k\in \mathcal{T}\}\subset\mathbb{R}$ satisfy $\min_{k\in \mathcal{T}}\operatorname{dist}_M(y,\gamma_k)\ge \sigma'$, $\max_{k\in \mathcal{T}}\operatorname{dist}_M(\gamma_k,k)\le q<1/2$, $\max_{k\in \mathcal{T}}\operatorname{dist}_M(u_k,k)\le p<1/2$, and $\max_{k\in \mathcal{T}}\operatorname{dist}_M(u_{k},\gamma_k)\le p-\frac14+\epsilon$. For $y\in \mathbb{R}$, let $\mathcal{T}_1=\Big\{k\in\mathcal{T} \mid \operatorname{dist}_M(y,k)\le q+\frac12\Big\}$. Then
\[ 
\Bigg|\prod_{k\in \mathcal{T}}\frac{\sin\frac{\pi(y-u_k)}{M}}{\sin\frac{\pi(y-\gamma_k)}{M}}\Bigg|
\lesssim \Big(1+\frac{\pi}{2}\Big)^{|\mathcal{T}_1|}\, \Bigl(1+\frac{p+\epsilon-1/4}{\sigma'}\Bigr)^{|\mathcal{T}_1|} e^{ (p+\epsilon-\frac14) (4+\pi)} M^{2p+2\epsilon-1/2}. \]
\end{lemma} 
\begin{proof}
For each $k\in\mathcal{T}$, we define
$P_k(y):=\frac{\sin\frac{\pi(y-u_k)}{M}}{\sin\frac{\pi(y-\gamma_k)}{M}}$, $y\in\mathbb{R}$. 
Note that $|\sin t|\le|t|$ for all $t\in\mathbb{R}$, $\frac{1}{\sin t}\le 1+\frac{1}{t}$ for any $0<t\le \frac{\pi}{2}$ (as the function $(1+t)\sin(t) -t$ is strictly increasing on $(0,\frac{\pi}{2}]$ with  a function value $0$ at $t=0$) and
\[ 
\sin\frac{\pi(y-u_k)}{M}-\sin\frac{\pi(y-\gamma_k)}{M} =-2\cos\frac{\pi(2y-\gamma_k-u_k)}{2M}\sin\frac{\pi(\gamma_k-u_k)}{2M}. 
\]
Based on the above facts and assumptions, it follows that
\begin{align}\label{eqPky}
|P_k(y)|
=\Bigg|1+\frac{\sin\frac{\pi(y-u_k)}{M}- \sin\frac{\pi(y-\gamma_k)}{M}}{\sin\frac{\pi(y-\gamma_k)}{M}}\Bigg| \nonumber
&\le 1+\frac{\pi \operatorname{dist}_M(\gamma_k,u_k)}{M}\frac{1}{|\sin\frac{\pi(y-\gamma_k)}{M}|} \nonumber\\
&\le 1+\frac{\pi (p+\epsilon-\frac{1}{4})}{M}\Big(1+\frac{M}{\pi\operatorname{dist}_M(y,\gamma_k)}\Big) \nonumber \\
&= 1+\frac{p+\epsilon-\frac14}{\operatorname{dist}_M(y,\gamma_k)}+\frac{\pi (p+\epsilon-\frac14)}{M}.
\end{align} 
Define $t_k:=\operatorname{dist}_M(y,k)$ for each $k\in\mathcal{T}$.
From $\operatorname{dist}_M(\gamma_k,k)\le q$ and the triangle inequality, we get $\operatorname{dist}_M(y,\gamma_k)\ge\operatorname{dist}_M(y,k)-\operatorname{dist}_M(k,\gamma_k)\ge t_k-q$. Partition the index set $\mathcal{T}$ into two subsets as follows:
\[ 
\mathcal{T}_1=\Big\{k\in\mathcal{T} \mid t_k\le q+\frac12\Big\}\mbox{ and } \mathcal{T}_2=\Big\{k\in\mathcal{T}\mid t_k>q+\frac12\Big\}. 
\] 

\textbf{Case 1:} For $k\in \mathcal{T}_1$, by $|y-\gamma_k|>\sigma'$, $p+\epsilon-\frac14<\frac{M}{2}$ and \eqref{eqPky}, we have 
\[
|P_k(y)|\le 1+\frac{p+\epsilon-\frac14}{\sigma'}+\frac{\pi}{2} \le\Bigl(1+\frac{p+\epsilon-\frac14}{\sigma'}\Bigr)\Big(1+\frac{\pi}{2}\Big). 
\]
Then
\begin{equation}\label{283} 
\prod_{k\in\mathcal{T}_1}|P_k(y)| \le \Big(1+\frac{\pi}{2}\Big)^{|\mathcal{T}_1|}\Bigl(1+\frac{p+\epsilon-1/4}{\sigma'}\Bigr)^{|\mathcal{T}_1|} . \end{equation} 

\textbf{Case 2:} For $k\in\mathcal{T}_2$, we have $\operatorname{dist}_M(y,\gamma_k)\ge t_k-q>\frac12$ and
$\frac{p+\epsilon-\frac14}{\operatorname{dist}_M(y,\gamma_k)}\le\frac{p+\epsilon-\frac14}{t_k-q}$. 
Using $\log(1+u)\le u$ for any $u\ge0$ and \eqref{eqPky}, we arrive at 
\[ 
\log|P_k(y)|\le\frac{p+\epsilon-\frac14}{t_k-q}+\frac{\pi (p+\epsilon-\frac14)}{M}. 
\] 
Summing over $k\in\mathcal{T}_2$  yields 
\begin{equation}\label{eqsumL}
\sum_{k\in\mathcal{T}_2}\log|P_k(y)| 
\le \pi( p+\epsilon-\frac{1}{4})+(p+\epsilon-\frac{1}{4})\sum_{k\in\mathcal{T}_2}\frac{1}{t_k-q}.
\end{equation}
Note that 
\begin{equation}\label{eqsumL1}
\sum_{k\in\mathcal{T}_2}\frac{1}{t_k-q}\le2\sum_{m=0}^{\lfloor \frac{M}{2}\rfloor}\frac{1}{m+1/2}\le 4+2\log M. 
\end{equation}
Combining \eqref{eqsumL} and \eqref{eqsumL1}, we have
\begin{equation}\label{284} 
\prod_{k\in\mathcal{T}_2}|P_k(y)| \lesssim e^{(p+\epsilon-\frac14)(4+\pi)}\,M^{2p+2\epsilon-\frac12}. 
\end{equation} 
Multiplying \eqref{283} and \eqref{284} gives the desired inequality. 
\end{proof}

\begin{lemma}\label{L25}
Let \(M\geq 2\) be an integer and $0<\eta<1/2$. Assume that $|v_k-k|<\frac14$ holds for all $k\in \mathbb{Z}_{M}$. If there exist $0<d\le\frac{M}{2}$, $\eta>0$, and an index $j\in\mathbb{Z}_{M}$ such that $x\in\mathbb{R}$ satisfies $\operatorname{dist}_M(x,v_{j})\le d$ and $\min_{k\in\mathbb{Z}_M,k\neq j}\operatorname{dist}_M(x,v_k)>\eta$, then
\begin{equation}\label{Prodeq}
\Big|\prod_{k\in\mathbb{Z}_M\setminus\{j\}}
\frac{
\sin\frac{\pi(v_{j}-v_k)}{M}}{\sin\frac{\pi(x-v_k)}{M}}
\Big|
\lesssim \frac{ (1+d)^3}{\eta}.
\end{equation}
\end{lemma}
\begin{proof} Its proof is given in Appendix C.
\end{proof}

\bigskip

{\textbf{Proof of Theorem \ref{Theorem}}}.

Recall that $F_{\Lambda}$ and $F_{\Lambda_1}$ are defined as in \eqref{FLambdax}.
Then, 
\[
|F_{\Lambda}'(\lambda_m)| = \frac{\pi}{M}\Bigg| \prod_{k\in\mathbb{Z}_M\setminus\{m\}} \sin \frac{\pi(\lambda_m \!-\! \lambda_k)}{M} \Bigg|
\text{ and }
|F_{\Lambda_1}'(v_m)| = \frac{\pi}{M}\Bigg| \prod_{k\in\mathbb{Z}_M\setminus\{m\}} \sin \frac{\pi(v_m \!-\! v_k)}{M} \Bigg|.
\]
Notice that 
\[
\Bigg|\frac{F_{\Lambda_1}'(v_m)}{F_{\Lambda}'(\lambda_m)}\Bigg|
\!=\!\Bigg|\prod_{k\in\mathbb{Z}_M\setminus\{m\}}
\frac{\sin\frac{\pi(v_m-v_k)}{M}}{
\sin\frac{\pi(\lambda_m-\lambda_k)}{M}}
\Bigg|\!=\!\Bigg|\prod_{k\in\mathbb{Z}_M\setminus\{m\}}
\frac{\sin\frac{\pi(\lambda_m-v_k)}{M}}{
\sin\frac{\pi(\lambda_m-\lambda_k)}{M}}
\Bigg|\Bigg|\prod_{k\in\mathbb{Z}_M\setminus\{m\}}
\frac{\sin\frac{\pi(v_m-v_k)}{M}}{
\sin\frac{\pi(\lambda_m-v_k)}{M}}
\Bigg|.
\]
Recalling \eqref{tiaojian} in Lemma \ref{lem11} and  choosing $\mathcal{T}=\mathbb{Z}_M\setminus\{m\}$, $\Gamma_1=\Lambda$, $\Gamma_2=\Lambda_1$, $y=\lambda_m$, $\sigma'=1-2L$, and $q=L$ in Lemma \ref{L23}, we have 
\begin{align*}
\Bigg|\prod_{k\in\mathbb{Z}_M\setminus\{m\}}
\frac{\sin\frac{\pi(\lambda_m-v_k)}{M}}{
\sin\frac{\pi(\lambda_m-\lambda_k)}{M}}
\Bigg| 
 \lesssim\Bigl(1+\frac{L+\epsilon-1/4}{1-2L}\Bigr) 
    e^{(L+\epsilon-\frac14)(4+\pi)}\,M^{2L+2\epsilon-\frac12}. 
\end{align*}
By Lemma \ref{L25} with $\eta=1/4$ (since \(0<L<1/2\), $\min_{j,j\neq m }\operatorname{dist}_{M}(\lambda_m,v_j)\ge 1/4$) and $d=L$, 
\[
\Bigg|\prod_{k\in\mathbb{Z}_M\setminus\{m\}}
\frac{\sin\frac{\pi(v_m-v_k)}{M}}{
\sin\frac{\pi(\lambda_m-v_k)}{M}}
\Bigg|
\lesssim 1.
\]
Therefore,
\begin{align}\label{gs2222222}
\Bigg|\frac{F_{\Lambda_1}'(v_m)}{F_{\Lambda}'(\lambda_m)}\Bigg|
&\lesssim \Big(1+\frac{\pi}{2}\Big)^{2\lfloor L \rfloor+1}\Bigl(1+\frac{L+\epsilon-1/4}{1-2L}\Bigr)^{2\lfloor L \rfloor+1}
 e^{(L+\epsilon-\frac14)(4+\pi)}M^{2L+2\epsilon-\frac12}\nonumber\\
&\lesssim\Bigl(1+\frac{L+\epsilon-1/4}{1-2L}\Bigr)M^{2L+2\epsilon-\frac12},
 \end{align}
where we have used $\lfloor L \rfloor=0$ and $L+\epsilon-\frac14\le \frac12+\frac18-\frac14=\frac38$.
Recall \eqref{tiaojian} in Lemma \ref{lem11}. By choosing $\mathcal{T}=\{k\in\mathbb{Z}_M\mid  \lambda_k\ne v_k\}$, $\Gamma_1=\Lambda_1$, $\Gamma_2=\Lambda$, $y=j$, $q=1/4$, $p=L$, $\sigma'=1/8$ in Lemma \ref{L23}, we have 

\begin{equation}
|V(j)|=\Bigg|\frac{F_{\Lambda}(j)}{F_{\Lambda_1}(j)}\Bigg|
 =\Bigg|\prod_{k\in \mathcal{T}}
\frac{\sin\frac{\pi(j-\lambda_k)}{M}}{
\sin\frac{\pi(j-v_k)}{M}}
\Bigg|\lesssim \Bigl(1+\frac{L+\epsilon-1/4}{1/8}\Bigr)
 e^{(L+\epsilon-\frac14)(4+\pi)}\,M^{2L+2\epsilon-\frac12}\nonumber
\end{equation}
where we have used $\lfloor L \rfloor=0$ and $L+\epsilon-\frac14\le\frac38$.
Similarly,   
\[
|V_m(j)|\lesssim (L+1)M^{2L+2\epsilon-1/2}.
\]
Then 
\begin{equation}\label{VsV}
|V_m(j)|+|V(j)|\lesssim (L+1)M^{2L+2\epsilon-1/2}.
\end{equation}
Combining \eqref{VsV} and \eqref{gs2222222}, we have 
\begin{equation}\label{zuizhongsss} 
 \sup_{j,m\in \mathbb{Z}_M}\Big( |V(j)|+|V_{m}(j)|\Big)\Bigg|
 \frac{F'_{\Lambda_1}(v_m)}{F'_{\Lambda}(\lambda_m)}\Bigg|\lesssim \Bigl(1+\frac{L+\epsilon-1/4}{1-2L}\Bigr)
 M^{4L+4\epsilon-1}.
 \end{equation} 
 By Lemma \ref{LemmaBsigmaL}, \eqref{eqBsigmaLFinal}, and \eqref{zuizhongsss}, 
 \begin{align*}
\|\bA_{\sigma,L}^{-1}\|_2=M^{-\frac12}\|\bB_{\sigma,L}\|_2
&\lesssim   \frac{1}{\epsilon} \Bigl(1+\frac{L+\epsilon-1/4}{1-2L}\Bigr)
 M^{4L+4\epsilon-3/2},
 \end{align*}
Choosing $\epsilon=\frac{\log 2}{8}\frac{1}{\log M}\le \frac18$, the above inequality yields the desired conclusion.

\bigskip

{\textbf{Proof of Corollary \ref{coro}}}.
The proof of the corollary is inspired by the proof of Theorem 3.1 in \cite{yu2023on}.
Since $f$ has $a$ derivatives,  by the Jackson theorem (see, for instance \cite[Theorem 41]{Meinardus1967}), there exists an $N$-order trigonometric polynomial \( p\)  such that $\|f-p\|_{L^{\infty}([-\pi,\pi])}\lesssim N^{-a}$. Let 
\[
q(x)=p(x)-\tilde{t}_N(x):=\sum_{j=-N}^{N} a_j \exp(- ijx).
\]
Let $\bA_{L}$ be defined as in \eqref{DefAL}. Note that the vector \(\bq:=(q(\tilde{x}_k): |k|\le N)^{\top}=\bA_{L}(a_k: |k|\le N)^{\top}\)
and $(\sum_{|k|\le N} |a_k|^2)^{1/2}=\|q\|_{L^{2}([-\pi,\pi])}$.

By \eqref{AL},  we have 
$\|\bA_{L}^{-1}\|_{2}\lesssim \frac{N^{4L-\frac{3}{2}}\log M}{1-2L}$ and the $L^2$-based Marcinkiewicz–Zygmund-type inequality 
\[
   \|q\|_{L^{2}([-\pi,\pi])}=\|\bA_{L}^{-1}\bq\|_2\le \|\bA_{L}^{-1}\|_2\|\bq\|_2 \lesssim \frac{N^{4L-3/2}\log M}{1-2L}  \|\bq\|_2\lesssim \frac{N^{4L-a-1}\log M}{1-2L},
\]
where we have used the fact that 
\begin{align*}
\|\bq\|_2\le \sqrt{2N+1}\max_{|k|\le N}|p(\tilde{x}_k)-\tilde{t}_N(\tilde{x}_k)|
&=\sqrt{2N+1}\max_{|k|\le N}|p(\tilde{x}_k)-f(\tilde{x}_k)|\\
&\le \sqrt{2N+1}\|p-f\|_{L^{\infty}([-\pi,\pi])}\lesssim \sqrt{2N+1} N^{-a}.
\end{align*}
Consequently, we have
\begin{align*}
  \|f  -  \tilde{t}_N \|_{L^{2}([-\pi,\pi])}
&\le  \|f  -  p \|_{L^{2}([-\pi,\pi])}+ \| q \|_{L^{2}([-\pi,\pi])}\\
&\lesssim N^{-a}+ \frac{N^{4L-a-1}\log M}{1-2L} \\
&\lesssim  \frac{N^{4L- a-1}\log M}{1-2L},
\end{align*}
where we have used the inequality $\|f  -  p \|_{L^{2}([-\pi,\pi])}\lesssim \|f  -  p \|_{L^{\infty}([-\pi,\pi])}$.
Therefore, 
\[
 |I - \tilde{I}_N|\lesssim  \|f  -  \tilde{t}_N \|_{L^{2}([-\pi,\pi])}\lesssim \frac{N^{4L- a-1}\log M}{1-2L}.
\]

\section{Proof of Theorem \ref{T001}}\label{Section4}

To prove Theorem \ref{T001}, we first establish two auxiliary lemmas. The first lemma asserts that for any separated point set, one can construct two families of near-integer auxiliary points and one family of compensation points in the same interval so that \eqref{ddengshi} holds and a certain rational product of sine functions at each original node is bounded only in terms of the separation distance and the number of given points.

\begin{lemma}\label{L31}
Let \(k\ge1\), \(0<\delta\le1\), and let
$A=b_1<b_2<\cdots<b_k=B$ be a sequence of points in $\bR$ satisfying the minimal separation condition
\[
    \min_{i\ne j}\operatorname{dist}_{M}(b_i,b_j)>\delta
\]
where $M>0$ is an integer. Define the intervals
\[
 K:=[A-k,B+k],\quad
 J:=\Big[A-k-\frac34,B+k+\frac34\Big],
\]
and assume that
\begin{equation}\label{eq:local-length}
B-A<M/2\quad \text{and} \quad  |J|=B-A+2k+\frac32\le\frac{9M}{10}.
\end{equation}
Then there exist pairwise distinct integers
\[
    n_1,n_2,\dots,n_k,\quad m_1,m_2,\dots,m_k
    \in K\cap\mathbb{Z},
\]
and three families of real numbers
\[
    u_1,u_2,\dots,u_k,\quad \gamma_1,\gamma_2,\dots,\gamma_k,
    \quad c_1,c_2,\dots,c_k
    \in J,
\]
satisfying the following properties.
\begin{itemize}
\item[(i)] For each $i=1,2,\dots,k$
\[
    u_i\in\Big(n_i-\frac18,n_i+\frac18\Big),
    \quad
    \gamma_i\in\Big(m_i-\frac18,m_i+\frac18\Big).
\]
Moreover, there exists an absolute constant \(\rho_0>0\) such that
\[
    \rho_0<\operatorname{dist}(u_i,\mathbb Z)<\frac18,
    \quad
    \rho_0<\operatorname{dist}(\gamma_i,\mathbb Z)<\frac18.
\]
All points $u_i$ and $\gamma_i$ are mutually distinct, and their reference integers $n_i$ and $m_i$ are likewise mutually distinct.

\item[(ii)] There exists a number $\theta$ with $|\theta|<5/8$ such that
\begin{equation}\label{eq:common-shift}
 c_i=\gamma_i+\theta,\quad i=1,2,\dots,k.
\end{equation}
The points $c_i$ are non-integers, pairwise separated modulo $M$ by an
absolute positive constant, and disjoint from
$\{b_1,b_2,\dots,b_k\}$. Moreover,
\[
 |u_i-b_i|\le k,
 \quad |c_i-\gamma_i|<\frac58,
 \quad i=1,2,\dots,k.
\]

\item[(iii)] There exists an absolute constant \(c_0>0\) such that
\[
    \operatorname{dist}_M(b_t,u_i)>\frac{c_0}{k},
    \quad
    \operatorname{dist}_M(b_t,\gamma_i)>\frac{c_0}{k}
    ,\quad
    \operatorname{dist}_M(b_t,c_i)>\frac{c_0}{k}
\]
for all $t,i=1,2,\dots,k$, and
\begin{equation}\label{ddengshi}
    \sum_{i=1}^k(b_i-u_i)
    =
    \sum_{i=1}^k(\gamma_i-c_i).
\end{equation}

\item[(iv)] For every $t=1,2,\dots,k$, the following product estimate holds:
\begin{equation}\label{Leq0}
\Bigg|
\prod_{i=1}^{k}
\frac{\sin\frac{\pi(b_t-\gamma_i)}{M}}
     {\sin\frac{\pi(b_t-c_i)}{M}}
\prod_{\substack{1\le i\le k\\ i\neq t}}
\frac{\sin\frac{\pi(b_t-u_i)}{M}}
     {\sin\frac{\pi(b_t-b_i)}{M}}
\Bigg|
\le
C\Big(\frac{C}{\delta}\Big)^{k-1},
\end{equation}
where \(C>0\) is an absolute constant.
\end{itemize}
\end{lemma}

\begin{proof} Throughout the proof, $C$ denotes an absolute constant whose value may change from line to line.
Fix \(0<\rho_0<10^{-2}\). All perturbations are chosen from intervals
$[n\pm10\rho_0,n\pm11\rho_0]$, so each point is non-integral, within
$1/8$ of an integer, and at distance at least $\rho_0$ from all integers.
There exists an absolute constant $c_0>0$ such that every interval
$I\subset\mathbb R$ of length $\rho_0$ contains a point $x$ satisfying
\begin{equation}\label{Leq1}
 \operatorname{dist}_M(x,b_t)>\frac{c_0}{k}
 \quad\text{for all }1\le t\le k.
\end{equation}
Indeed, the $c_0/k$-neighbourhoods of the $k$ points $b_t$ cover total
length at most $2c_0$, so taking $c_0<\rho_0/4$ leaves a point outside
their union.

Since $\operatorname{dist}_M(a,b)\le|a-b|$, the modular separation
assumption implies $|b_{j+1}-b_j|>\delta$ for every $j$. Consequently,
\begin{equation}\label{Leq2}
 |b_t-b_i|
 =\sum_{j=\min(t,i)}^{\max(t,i)-1}(b_{j+1}-b_j)
 >\delta|t-i|.
\end{equation}

We now construct the reference integers in a balanced way. For each $i=1,2,\dots,k$, define
\[
 \ell_i:=\lfloor b_i\rfloor-(k-i)
 \mbox{  and  }
 r_i:=\lceil b_i\rceil+i-1.
\]
Both sequences are strictly increasing, $\ell_i,r_i\in K\cap\mathbb Z$,
and
\[
 \sum_{i=1}^k\ell_i\le\sum_{i=1}^k b_i
 \le\sum_{i=1}^k r_i.
\]
Let $T$ be an integer nearest to $\sum_{i=1}^k b_i$. There exist strictly
increasing integers $n_i$ satisfying
\[
 \ell_i\le n_i\le r_i,
 \quad
 \sum_{i=1}^k n_i=T.
\]
Indeed, starting from $\ell_i$ and increasing by one the largest
coordinate that can still be increased preserves strict ordering and
produces every integer sum between $\sum_{i=1}^k \ell_i$ and $\sum_{i=1}^k r_i$.
Hence
\begin{equation}\label{eq:small-total-drift}
 \Big|\sum_{i=1}^k(n_i-b_i)\Big|\le\frac12
 \mbox{ and } |n_i-b_i|<k.
\end{equation}

Choose $u_i$ in one of the intervals
$(n_i-11\rho_0,n_i-10\rho_0)$ or
$(n_i+10\rho_0,n_i+11\rho_0)$, taking the side pointing toward $b_i$
when $n_i\ne b_i$. Using \eqref{Leq1}, we can impose
$\operatorname{dist}_M(u_i,b_t)>c_0/k$ for every $t$, while retaining
$|u_i-b_i|\le k$. Set
\[
 \theta:=\frac1k\sum_{i=1}^k(u_i-b_i).
\]
By \eqref{eq:small-total-drift} and $|u_i-n_i|<1/8$,
\[
 |\theta|\le\frac{1}{2k}+\frac18\le\frac58.
\]
A harmless change of one $u_i$ makes the last inequality strict.

The interval $K$ has length at least $2k$ and therefore contains at
least $2k$ integers. After removing $n_1,\dots,n_k$, choose distinct
integers $m_1,\dots,m_k\in K\cap\mathbb Z$. For each $i$, choose
$\gamma_i$ from
\[
 (m_i-11\rho_0,m_i-10\rho_0)
 \cup(m_i+10\rho_0,m_i+11\rho_0)
\]
so that both $\gamma_i$ and $\gamma_i+\theta$ have distance larger than
$c_0/k$ from every $b_t$, and $\gamma_i+\theta\notin\mathbb Z$. This is
possible after decreasing $c_0$, since the forbidden subsets have total
length at most $4c_0$. Define $c_i:=\gamma_i+\theta$. Then
\eqref{eq:common-shift}, \eqref{ddengshi}, and the distance estimates in
(iii) follow immediately. Moreover,
\[
 |c_i-c_j|=|\gamma_i-\gamma_j|\ge1-22\rho_0>1-\frac{22}{100}\ge\frac34
\]
for distinct reference integers in the chosen lift. Together with
\eqref{eq:local-length}, this also gives an absolute lower bound for the modular separation of the $c_i$. This proves (i)--(iii).

To prove conclusion (iv), fix $t\in\{1,2,\dots,k\}$.
Note that
\begin{equation}\label{sincbound}
\frac{2}{\pi}|\theta|
\le \sin|\theta|
\le |\theta|\;
\mbox{ for any }
0\le|\theta|\le\frac{\pi}{2}.
\end{equation}
Using the separation condition $\min_{i\neq j}|c_i-c_{j}|\gtrsim 1$, \eqref{sincbound}, the 1‑Lipschitz continuity of the sine function, 
and the distance estimate in (iii), we have
\begin{align}\label{eq:gamma-c-product}
 \prod_{i=1}^k
 \Bigg|\frac{\sin\frac{\pi(b_t-\gamma_i)}M}
              {\sin\frac{\pi(b_t-c_i)}M}\Bigg|
&= \prod_{i=1}^k\Bigg|\frac{\sin\frac{\pi\operatorname{dist}_M(b_t,\gamma_i)}{M} }
              {\sin\frac{\pi\operatorname{dist}_M(b_t,c_i)}{M}}\Bigg|
\le \prod_{i=1}^k\Bigg|1+\frac{ \frac{\pi\operatorname{dist}_M(\gamma_i,c_i)}{M} }
              {\sin\frac{\pi\operatorname{dist}_M(b_t,c_i)}{M}}\Bigg|  \nonumber\\           
&\le\prod_{i=1}^k
\Bigg|1+ \frac{\pi}{2} \frac{ |\theta|}
          {\operatorname{dist}_M(b_t,c_i)}\Bigg|\nonumber\\
&\lesssim k\prod_{i\in \mathbb{Z}\cap[1,k]\backslash \{i\mid |b_t-c_i|<1/8\} }
 \Big|1+\frac{1}
              {\operatorname{dist}_M(b_t,c_i)}\Big|
 \le C^k,
\end{align}
where we have used the fact $\frac{\pi }{2}|\theta|\le\frac{\pi }{2}\frac{5}{8}<1$.

Using $|u_i-b_i|\le k$ and \eqref{Leq2}, \eqref{sincbound} gives for $i\ne t$,
\[
\Bigg|
 \frac{\sin\frac{\pi(b_t-u_i)}M}
      {\sin\frac{\pi(b_t-b_i)}M}
 \Bigg|
 \le \frac{\frac{\pi|b_t-u_i|}M}{\frac{2}{\pi}\frac{\pi|b_t-b_i|}M}
=\frac{\pi}{2}\frac{|b_t-b_i+b_i-u_i|}{|b_t-b_i|}
 \le  \frac{\pi}{2}\Big(1+\frac{k}{\delta|t-i|}\Big).
\]

Since the same distance $|t-i|$ can occur on both sides of $t$,
\[
 \prod_{i\ne t}\Big(1+\frac{k}{\delta|t-i|}\Big)
 \le \delta^{-(k-1)}
 \Bigg[\prod_{j=1}^{k-1}\Big(1+\frac{k}{j}\Big)\Bigg]^2=\delta^{-(k-1)}\binom{2k-1}{k-1}^{\!2}
 \le \delta^{-(k-1)}16^k.
\]
Combining this estimate with \eqref{eq:gamma-c-product}, and enlarging
$C$, proves \eqref{Leq0}.
\end{proof}

\begin{lemma}\label{L32} Let  $0<\delta\le1$.
Let \(M\) be a positive integer, and let
\[
 [A_j,B_j]\subset[0,M),\quad 1\le j\le r,
\]
be closed intervals ordered cyclically from left to right. Write
\[
 A_j=b_{j,1}<b_{j,2}<\cdots<b_{j,k_j}=B_j,
\]
and set
\[
 s:=\sum_{j=1}^r k_j,
 \quad
 \kappa:=\max_{1\le j\le r}k_j.
\]
Assume that $s<M/2$ and 
\[
 \operatorname{dist}_M(b_{j,i},b_{q,p})>\delta
 \quad
 \text{for} \quad  (j,i)\ne(q,p).
\]
Using the cyclic convention \(A_{r+1}:=A_1+M\), suppose that,
for every \(j\),
\begin{equation}\label{eq:new-local-size}
B_j-A_j<M/2 ,\quad B_j-A_j+2k_j+\frac32\le\frac{9M}{10},
\end{equation}
and
\begin{equation}\label{eq:new-gap}
 A_{j+1}-B_j\ge k_j+k_{j+1}+2.
\end{equation}

Let
\[
 \mathcal B
 :=
 \{b_{j,i}:1\le j\le r,\ 1\le i\le k_j\}.
\]
Then there exist two sets of \(M\) pairwise distinct points modulo \(M\),
\[
 \Lambda=\{\lambda_m\mid m\in\mathbb Z_M\},
 \quad
 \Lambda_1=\{v_m\mid m\in\mathbb Z_M\},
\]
which may be indexed so that
\[
 \mathcal B=\{\lambda_m\mid m\in\mathbb Z_s\},
\]
and the following properties hold.

\begin{itemize}
\item[(i)]
For every \(\nu\in\mathbb Z_M\), the set \(\Lambda_1\) contains
exactly one point in
\[
 \Big(\nu-\frac18,\nu+\frac18\Big)\pmod M.
\]
Moreover, for some absolute constant \(c>0\),
\begin{equation}\label{eq:L32-near-grid}
 c<\operatorname{dist}(v,\mathbb Z)<\frac18,
 \quad v\in\Lambda_1,
\end{equation}
and
\begin{equation}\label{eq:L32-separation}
 \operatorname{dist}_M(\mathcal B,\Lambda_1)
 \ge\frac{c}{\kappa},
 \quad
 \operatorname{dist}_M(\lambda_m,v_m)\le \kappa,
 \quad m\in\mathbb Z_s.
\end{equation}

\item[(ii)]
Define
\[
 V(x)
 :=
 \prod_{k\in\mathbb Z_M}
 \frac{\sin\frac{\pi(x-\lambda_k)}{M}}
      {\sin\frac{\pi(x-v_k)}{M}},
 \quad
 V_m(x)
 :=
 \prod_{k\in\mathbb Z_M\setminus\{m\}}
 \frac{\sin\frac{\pi(x-\lambda_k)}{M}}
      {\sin\frac{\pi(x-v_k)}{M}},
\]
where common factors are canceled. Then
\begin{equation}\label{eq:L32-V-bound}
 \max_{x\in\mathbb Z_M}|V(x)|
 +
 \max_{\substack{m\in\mathbb Z_s\\x\in\mathbb Z_M}}
 |V_m(x)|
 \le C^\kappa,
\end{equation}
where \(C>0\) is an absolute constant.

\item[(iii)]
For every \(m\in\mathbb Z_s\),
\begin{equation}\label{eq:L32-derivative-bound}
 \Big|
 \frac{F_{\Lambda_1}'(v_m)}
      {F_\Lambda'(\lambda_m)}
 \Big|
 \le
 C\Big(\frac{C}{\delta}\Big)^{\kappa-1},
\end{equation}
where
\[
 F_\Gamma(x)
 :=
 \prod_{\gamma\in\Gamma}
 \sin\frac{\pi(x-\gamma)}{M}.
\]
\end{itemize}
\end{lemma}

\begin{proof}
Throughout the proof, \(C,c>0\) are absolute constants  whose value may change from line to line. We divide the argument into three steps.

\medskip
\noindent
\textbf{Step 1. Construction.}
For each \(1\le j\le r\), let
\[
 K_j=[A_j-k_j,B_j+k_j],
 \quad
 J_j=\Big[A_j-k_j-\frac34,B_j+k_j+\frac34\Big].
\]
For \(1\le j\le r\), define the \(j\)-th block by
\[
\mathcal B_j
=
\{b_{j,i}:1\le i\le k_j\}
\subset [A_j,B_j].
\]
By \eqref{eq:new-gap}, we have
\[
 \inf J_{j+1}-\sup J_j\ge\frac12,
 \quad
 \inf K_{j+1}-\sup K_j\ge2.
\]
 Hence the intervals \(J_j\) are cyclically
disjoint, and the integer classes selected by Lemma~\ref{L31} in
different blocks are distinct modulo \(M\).

Applying Lemma~\ref{L31} to every block, we obtain points
\(u_{j,i},w_{j,i},c_{j,i}, 1\le j\le r, 1\le i\le k_j\) and distinct integer classes
\(n_{j,i},m_{j,i}, 1\le j\le r, 1\le i\le k_j\) such that
\[
 u_{j,i}\in\Big(n_{j,i}-\frac18,n_{j,i}+\frac18\Big),
 \quad
 w_{j,i}\in\Big(m_{j,i}-\frac18,m_{j,i}+\frac18\Big),
\]
\begin{equation}\label{eq:L32-local-data}
 \rho_0<\operatorname{dist}(u_{j,i},\mathbb Z)<\frac18,
 \rho_0<\operatorname{dist}(w_{j,i},\mathbb Z)<\frac18,
 |u_{j,i}-b_{j,i}|\le k_j,
 |w_{j,i}-c_{j,i}|<\frac58.
\end{equation}
Moreover,
\begin{equation}\label{eq:L32-balance}
 \sum_{i=1}^{k_j}(u_{j,i}-b_{j,i})
 +\sum_{i=1}^{k_j}(w_{j,i}-c_{j,i})=0,
\end{equation}
and
\begin{equation}\label{eq:L32-local-separation}
 \operatorname{dist}_M
 \bigl(\mathcal B,\{u_{j,i},w_{j,i},c_{j,i}\mid  1\le j\le r, 1\le i\le k_j\}\bigr)
 \gtrsim\frac{1}{\kappa}.
\end{equation}
The points \(c_{j,i}\) are pairwise separated modulo \(M\) by an
absolute constant and are disjoint from \(\mathcal B\).

Let \(\mathcal N:=\{n_{j,i}\mid 1\le j\le r, 1\le i\le k_j\}\cup\{m_{j,i}\mid  1\le j\le r, 1\le i\le k_j\} \), then $|\mathcal N|=2s$. Let 
\(\mathcal R=\mathbb Z_M\setminus\mathcal N\).  For each
\(l\in\mathcal R\), choose
\[
 \mu_l\in\Big(l+\frac1{16},l+\frac18\Big)\pmod M
\]
so that
\begin{equation}\label{eq:L32-filling}
 \operatorname{dist}_M(\mu_l,\mathcal B)\ge\frac{c}{\kappa},
 \quad
 \mu_l\notin\{c_{j,i}\mid  1\le j\le r, 1\le i\le k_j\}.
\end{equation}
Indeed, this interval meets at most one block and hence at most
\(\kappa\) forbidden \(c/\kappa\)-neighbourhoods, whose total length
is at most \(2c<1/16\).

Now define
\[
 \Lambda
 =\mathcal B\cup\{c_{j,i}\mid  1\le j\le r, 1\le i\le k_j\}\cup\{\mu_l\mid l\in\mathcal R\},
\]
\[
 \Lambda_1
 =\{u_{j,i}\mid  1\le j\le r, 1\le i\le k_j\}\cup\{w_{j,i}\mid  1\le j\le r, 1\le i\le k_j\}
  \cup\{\mu_l\mid l\in\mathcal R\}.
\]
Both sets contain \(M\) distinct points modulo \(M\).  Pair
\[
 b_{j,i}\leftrightarrow u_{j,i},
 \quad
 c_{j,i}\leftrightarrow w_{j,i},
 \quad
 \mu_l\leftrightarrow\mu_l,
\]
and index the pairs so that
\(\mathcal B=\{\lambda_m\mid m\in\mathbb Z_s\}\).  Thus every integer
class contains exactly one point of \(\Lambda_1\) within \(1/8\), and
\eqref{eq:L32-local-data}, \eqref{eq:L32-local-separation}, and
\eqref{eq:L32-filling} prove~\emph{(i)}.

\medskip
\noindent
\textbf{Step 2. A common block estimate.}
For \(\phi\in C^1(I)\), where \(I\) contains the \(q\)-th block and
its auxiliary points, set
\[
 \mathcal D_q(\phi)
 =\sum_{p=1}^{k_q}\bigl[\phi(u_{q,p})-\phi(b_{q,p})\bigr]
 +\sum_{p=1}^{k_q}\bigl[\phi(w_{q,p})-\phi(c_{q,p})\bigr].
\]
By \eqref{eq:L32-balance}, subtracting any constant from \(\phi'(I)\)
and integrating along the above segments gives
\begin{equation}\label{eq:L32-balanced-estimate}
 |\mathcal D_q(\phi)|=\bigg|\sum_{p=1}^{k_q} \int_{b_{q,p}}^{u_{q,p}}(\phi'(t)-C)dt +\sum_{p=1}^{k_q} \int_{c_{q,p}}^{w_{q,p}}(\phi'(t)-C)dt\bigg|
 \le2k_q^2\operatorname{osc}_I(\phi'),
\end{equation}
where 
\[
\operatorname{osc}_{I}(f)
:=
\sup_{x,y\in I}|f(x)-f(y)|.
\]
We also use
\begin{equation}\label{eq:L32-sine-ratio}
 \Big|
 \frac{\sin\frac{\pi(x-z)}M}{\sin\frac{\pi(x-y)}M}
 \Big|
= \Big|
 \frac{\sin\frac{\pi\operatorname{dist}_M(x,z)}M}{\sin\frac{\pi\operatorname{dist}_M(x,y)}M}
 \Big|
 \le 1+\frac{\pi}{2}\frac{\operatorname{dist}_M(y,z)}{\operatorname{dist}_M(x,y)}
 \le1+\frac{C|z-y|}{\operatorname{dist}_M(x,y)},
\end{equation}
whenever the denominator is nonzero, where we have used \eqref{sincbound} and the 1‑Lipschitz continuity of the sine function in the first inequality.

For \(a\in\mathbb R\), put
\[
 \Phi_a(y)=\log\Big|\sin\frac{\pi(a-y)}M\Big|.
\]
Fix $a$, declare a block {\textit near} if it contains \(a\pmod M\), or
if its   ${\textrm dist}_M$-distance at most \(10\kappa\) from
\(a\), otherwise call the block {\textit far}.  By \eqref{eq:new-gap},
\begin{equation}\label{eq:L32-near-mass}
 \sum_{q\ \mathrm{near}}k_q\le C\kappa.
\end{equation}
For the far blocks,
\[
 |\Phi_a''(a+t)|
 \le C\Big(\frac1{t^2}+\frac1{(M-t)^2}\Big),
 \quad 0<t<M.
\]
Hence \eqref{eq:L32-balanced-estimate} and disjointness yield
\begin{equation}\label{eq:L32-far-sum}
 \sum_{q\ \mathrm{far}}|\mathcal D_q(\Phi_a)|\lesssim \kappa^2 \int_{8 \kappa}^{M-8\kappa}\Big(\frac{1}{t^2}+\frac{1}{(M-t)^2}\Big)dt\lesssim \kappa.
\end{equation}
When \(M\le20\kappa\), there are no far blocks; otherwise we just need to integrate $|\Phi_a''(a+t)|$ over \([8\kappa,M-8\kappa]\).

We shall also use the following elementary consequence of ordering
points by their distance from a fixed point.  If \(N\le C\kappa\)
points $\{y_l\}_{l=1}^{N}$ lie within \(1/8\) of distinct integer classes and stay an
absolute distance from \(\mathbb Z\), then, for every integer \(n\),
\begin{equation}\label{eq:L32-ordering-product}
 \prod_{l=1}^N
 \Big(1+\frac{C\kappa}{\operatorname{dist}_M(n,y_l)}\Big)
 \le C^\kappa.
\end{equation}
Indeed, after ordering, \(\operatorname{dist}_M(n,y_l)\ge c l\), and
an integral comparison completes the proof.

\medskip
\noindent
\textbf{Step 3. Product bounds.}
Fix \(x=b_{j,i}\).  Lemma~\ref{L31} bounds the contribution of the
\(j\)-th block by
\(
 C\Big(\frac{C}{\delta}\Big)^{k_j-1}.
\)
For a near block \(q\ne j\), the adjacent gap condition implies
\(\operatorname{dist}_M(x,b_{q,p})\ge k_q+2\).   By
\eqref{eq:L32-sine-ratio}, 
\(\big|\sin\frac{\pi(x-u_{q,p})}{M}/
      \sin\frac{\pi(x-b_{q,p})}{M}\big|\lesssim 1\) for \(q\neq j\).
It follows that 
\[ \Bigg|\prod_{q\neq j,q~{\textrm{near}}}\prod_{p=1}^{k_q}
 \frac{\sin\frac{\pi(x-u_{q,p})}{M}}
      {\sin\frac{\pi(x-b_{q,p})}{M}}\Bigg|\le C^{\kappa}.\]
For \(q\neq j\), since \(x=b_{j,i}\in J_j\), \(c_{q,p}\in J_q\), and the intervals \(J_q\) are cyclically separated by at least \(1/2\), we have
\(
\operatorname{dist}_M(x,c_{q,p})\ge \frac12.
\)
Hence, by \eqref{eq:L32-sine-ratio},   $\prod_{p=1}^{k_q}
 \frac{\sin\frac{\pi(x-w_{q,p})}{M}}
      {\sin\frac{\pi(x-c_{q,p})}{M}} $ in the near block  contribute at most \(C^{k_q}\) when $q\neq j$. By \eqref{eq:L32-near-mass},
\[  \bigg|\prod_{q\ \textrm{near},~q\neq j}\prod_{p=1}^{k_q}
 \frac{\sin\frac{\pi(x-w_{q,p})}{M}}
      {\sin\frac{\pi(x-c_{q,p})}{M}}\bigg|\le \prod _{q\ \mathrm{near},~q\neq j} C^{k_{q}}\le C^{\kappa}. \]
The far blocks
contribute at most \(C^\kappa\) by \eqref{eq:L32-far-sum}.   Therefore
\begin{equation}\label{eq:L32-product-at-B}
\begin{aligned}
 \Bigg|
 &\prod_{q=1}^r\prod_{p=1}^{k_q}
 \frac{\sin\frac{\pi(x-w_{q,p})}{M}}
      {\sin\frac{\pi(x-c_{q,p})}{M}} 
 \prod_{\substack{1\le q\le r,\ 1\le p\le k_q\\(q,p)\ne(j,i)}}
 \frac{\sin\frac{\pi(x-u_{q,p})}{M}}
      {\sin\frac{\pi(x-b_{q,p})}{M}}
 \Bigg|
 \le C\Big(\frac{C}{\delta}\Big)^{\kappa-1}.
\end{aligned}
\end{equation}

Now fix \(n\in\mathbb Z\). Similarly to the proof of \eqref{eq:L32-product-at-B}, by \eqref{eq:L32-near-grid}, we have
\begin{equation}\label{eq:L32-product-at-integers}
\begin{aligned}
 \Bigg|
 &\prod_{q=1}^r\prod_{p=1}^{k_q}
 \frac{\sin\frac{\pi(n-c_{q,p})}{M}}
      {\sin\frac{\pi(n-w_{q,p})}{M}} 
 \prod_{q=1}^r\prod_{p=1}^{k_q}
 \frac{\sin\frac{\pi(n-b_{q,p})}{M}}
      {\sin\frac{\pi(n-u_{q,p})}{M}}
 \Bigg|
 \le C^\kappa,
\end{aligned}
\end{equation}
and 
\begin{equation}\label{eq:L32-deleted-product}
\begin{aligned}
 \Bigg|
 &\prod_{q=1}^r\prod_{p=1}^{k_q}
 \frac{\sin\frac{\pi(n-c_{q,p})}{M}}
      {\sin\frac{\pi(n-w_{q,p})}{M}} 
 \prod_{\substack{1\le q\le r,\ 1\le p\le k_q\\(q,p)\ne(j,i)}}
 \frac{\sin\frac{\pi(n-b_{q,p})}{M}}
      {\sin\frac{\pi(n-u_{q,p})}{M}}
 \Bigg|
 \le C^\kappa.
\end{aligned}
\end{equation}

The factors corresponding to \(\mu_l\) cancel from \(V\) and \(V_m\).
Under the above pairing, \eqref{eq:L32-product-at-integers} and
\eqref{eq:L32-deleted-product} are exactly the bounds for \(V(n)\)
and \(V_m(n)\), respectively.  This proves \eqref{eq:L32-V-bound}.

Let \(m\in\mathbb Z_s\), say
\(\lambda_m=b_{j,i}\) and \(v_m=u_{j,i}\).  Then
\begin{align*}
 \Bigg|\frac{F_{\Lambda_1}'(v_m)}{F_\Lambda'(\lambda_m)}\Bigg|
 &=
 \Bigg|
 \prod_{k\ne m}
 \frac{\sin\frac{\pi(\lambda_m-v_k)}M}
      {\sin\frac{\pi(\lambda_m-\lambda_k)}M}
 \Bigg| \nonumber
 \Bigg|
 \prod_{k\ne m}
 \frac{\sin\frac{\pi(v_m-v_k)}M}
      {\sin\frac{\pi(\lambda_m-v_k)}M}
 \Bigg|.
\end{align*}
The first product is bounded by \eqref{eq:L32-product-at-B}.  For the
second, temporarily label by \(\widetilde v_\nu\) the unique point of
\(\Lambda_1\) lying within \(1/8\) of \(\nu\in\mathbb Z_M\), and let
\(\widetilde v_{\nu_0}=v_m\).  By the fact $s<M/2$ and
\eqref{eq:L32-separation},
\[
 \operatorname{dist}_M(\lambda_m,\widetilde v_{\nu_0})
 \le\kappa<\frac M2,
 \quad
 \min_{\nu\ne\nu_0}
 \operatorname{dist}_M(\lambda_m,\widetilde v_\nu)
 \ge\frac{c}{\kappa}.
\]
Lemma~3.2, with \(d=\kappa\) and \(\eta=c/(2\kappa)\), gives
\[
 \Bigg|
 \prod_{k\ne m}
 \frac{\sin\frac{\pi(v_m-v_k)}M}
      {\sin\frac{\pi(\lambda_m-v_k)}M}
 \Bigg|
 \le C\kappa^4.
\]
Combining this with \eqref{eq:L32-product-at-B} and absorbing
\(\kappa^4\le C^\kappa\) into the absolute constant proves
\eqref{eq:L32-derivative-bound}.
\end{proof}

 \bigskip

\textbf{Proof of Theorem \ref{T001}.}

We first verify that the assumptions of Lemma \ref{L32} are satisfied.
Recall  Assumption \ref{Assumption} and $r\ge2$, set   $\mathcal{C}_j =\{b_{j,i}\}_{i=1}^{k_j}\subset [A_j,B_j]$, where 
\(
B_j-A_j=\alpha_j<\frac{M}{3}.
\)
Since $M\ge4s$ and $k_j\le s-1$,
\[
B_j-A_j+2k_j+\frac32<\frac{9M}{10}.
\]
Moreover,
\[
A_{j+1}-B_j
\ge \beta_{j,j+1}
\ge k_j+k_{j+1}+2.
\]
Therefore, Lemma \ref{L32} is applicable, and $\mathcal B=\bigcup_{j=1}^{r} \mathcal{C}_j$.

By Lemma~\ref{L32}, there exist 
\(\Lambda=\{\lambda_m\mid m\in\mathbb Z_M\),
\(\Lambda_1=\{v_m\mid m\in\mathbb Z_M\),
\(
 \mathcal{B}=\{\lambda_m\mid m\in\mathbb Z_s\}
\), and if $\lambda_{m_{i,j}}=b_{i,j}$, then $v_{m,{i,j}}=u_{i,j}$.
By \eqref{eq:new-gap}, after translating all points by the same integer, we may assume that
all points lie in $[-1/2,M-1/2)$ (we still denote them by $\{\lambda_m\mid m\in\mathbb Z_M\},\{v_m\mid m\in\mathbb Z_M\}$) and $|\lambda_m-v_m|\le \kappa<M/2$. Thus the hypotheses of Lemma~\ref{LLLL27} are satisfied.
Moreover,
\begin{equation}\label{vvmvv}
\max_{x\in\mathbb Z_M}|V(x)|
 +
 \max_{\substack{x\in\mathbb Z_M\\m\in\mathbb Z_s}}
 |V_m(x)|
 \le C^\kappa,
\end{equation}
and
\begin{equation}\label{qqmvv}
 \max_{m\in\mathbb Z_s}
 \Bigg|
 \frac{F_{\Lambda_1}'(v_m)}
      {F_\Lambda'(\lambda_m)}
 \Bigg|
 \le
 C\Big(\frac C\delta\Big)^{\kappa-1}.
\end{equation}

Recall the $m\times s$ matrix $\bA_{\X}$ defined in Theorem
\ref{T001}, where $\X=\mathcal{B}$. Let $\bA_{\sigma,L}$ denote the
matrix associated with the frequency set $\Lambda$. This matrix $\bA_{\sigma,L}$ is formed by
augmenting $\bA_{\X}$ with $\bA_2$, that is,
\[
\bA_{\sigma,L}=[\bA_{\X},\bA_2],
\]
where
\[
\bA_2:=
[\exp(-2\pi i j\lambda_k/M):
j\in\mathbb{Z}_M,\ s\le k\le M-1].
\]
For any vector
\(
\ba=(\ba_1^{\top},{\textbf 0}^{\top})^{\top}
\)
with $\|\ba\|_2=1$, where $\ba_1\in\mathbb{C}^s$ and
${\textbf 0}\in\mathbb{C}^{M-s}$, we have
\[
\begin{aligned}
1
&=\|\ba_1\|_2
 =\|\ba\|_2=\|\bI_s\bA_{\sigma,L}^{-1}
      \bA_{\sigma,L}\ba\|_2\le
\|\bI_s\bA_{\sigma,L}^{-1}\|_2
\|\bA_{\sigma,L}\ba\|_2=
\|\bI_s\bA_{\sigma,L}^{-1}\|_2
\|\bA_{\X}\ba_1\|_2.
\end{aligned}
\]
Choosing the vector $\ba_1$ with $\|\ba_1\|_2=1$ such that
\[
\|\bA_{\X}\ba_1\|_2
=
\inf_{\|\bv\|_2=1}\|\bA_{\X}\bv\|_2,
\]
we conclude that
\[
\sigma_{\min}(\bA_{\X})
\ge
\frac{1}{\|\bI_s\bA_{\sigma,L}^{-1}\|_2}.
\]
Thus, it remains to estimate
$\|\bI_s\bA_{\sigma,L}^{-1}\|_2$.
By Remark \ref{RemarkKadec} and Lemma \ref{27}, we have 
\[ 
\|\bD^*\|_2\le2
\mbox{ and } \sup_{n\in\mathbb Z}\|\bD_n\|_2\le2.
\]
By Lemmas \ref{Lemmarcsingularvalue} and \ref{ExampleCnorm},
\[
\|\bC_n\circ \bD_0\|_2
\lesssim \frac{\kappa}{(|n|-1)^2},
\quad |n|\ge2.
\]
Therefore,
\[
\sum_{|n|\ge2}\|\bC_n\circ \bD_0\|_2\lesssim \kappa.\]
Combining Lemma \ref{LemmaBsigmaL}, Lemma \ref{LLLL27},
\eqref{vvmvv}, and \eqref{qqmvv}, we obtain
\[
\begin{aligned}
\|\bI_s\bA_{\sigma,L}^{-1}\|_2
&=
\frac{\|\bI_s\bB_{\sigma,L}\|_2}{\sqrt{M}}
\lesssim
\frac{C^{\kappa}}{\sqrt{M}}
\sup_{m\in\mathbb{Z}_s}
\Bigg|
\frac{F_{\Lambda_1}'(v_m)}
     {F_{\Lambda}'(\lambda_m)}
\Bigg|
\lesssim
\frac{C^{\kappa}}{\sqrt{M}}
\Big(\frac{C_3}{\delta}\Big)^{\kappa-1}
\lesssim
\frac{1}{\sqrt{M}}
\Big(\frac{C_5}{\delta}\Big)^{\kappa-1},
\end{aligned}
\]
where the last inequality follows from $\kappa\ge2$, after enlarging
the absolute constant $C_5$.

Recalling Assumption \ref{Assumption} and replacing $\delta$ by
$M\delta$, we obtain
\[
\sigma_{\min}(\bA_{\X})
\ge
\frac{1}{\|\bI_s\bA_{\sigma,L}^{-1}\|_2}
\gtrsim
\sqrt{M}
\Big(\frac{M\delta}{C_5}\Big)^{\kappa-1},
\]
which yields the desired conclusion.

\section{Concluding Remarks}\label{Section5}

The methodological contribution of this paper is a reduction of inverse nonuniform Fourier matrix estimates to spectral norm estimates for periodic nonuniform interpolation matrices.
The reduction applies to both clustered configurations and perturbations of an equispaced grid, although
the geometric constructions required in the two settings are different. For clustered nodes, a local family of
auxiliary nodes and compensation points compares a highly nonuniform configuration with a near-integer
interpolation system and yields a separation condition determined by adjacent cluster sizes. 
For perturbed grids, a comparison with a Kadec-stable configuration leads to consequences for $2$-norm Lebesgue constants and trigonometric quadrature. We believe that our framework also applies to singular value estimates under the classical minimum-separation condition. In that setting, however, the interpolation scheme used in Lemma \ref{L31} is not applicable. Since the lower bound established in \cite{aubel2019vandermonde} is already nearly optimal, we do not pursue this case in the present paper.

\appendix
\setcounter{section}{1}
\section*{Appendix A: The Pointwise Convergence of \eqref{gs0pre} on \texorpdfstring{\((-\pi,\pi)\)}{}}\label{AppendixA}

We shall prove that \eqref{gs0pre} holds pointwise for every
\(y\in(-\pi,\pi)\). To begin with, assume 
\[
j\notin\{\tau_{mn}\mid m\in\mathbb Z_M,\ n\in\mathbb Z\}.
\]
For each \(m\in\mathbb Z_M\), set
\(
a_m:=\frac{j-\lambda_m}{M}.
\)
Clearly, \(a_m\notin\mathbb Z\).
Then, by \eqref{psimn}, we have 
\begin{equation}\label{pointwise_series}
\sum_{n\in\mathbb Z}\psi_{mn}(j)e^{i\tau_{mn}y}
\!=\!\sum_{n\in\mathbb Z}\frac{F_{\Lambda}(j) e^{i(\lambda_m+nM)y}}{F'_{\Lambda}(\lambda_m+nM)(j-\lambda_m-nM)}
\!=\!
\frac{F_\Lambda(j)e^{i\lambda_m y}}
     {M F_\Lambda'(\lambda_m)}
\sum_{n\in\mathbb Z}
\frac{e^{inM(y+\pi)}}{a_m-n}.
\end{equation}
By \eqref{Fourierexpansion}, 
\[
 \frac{\pi e^{ia(t-\pi)}}{\sin(\pi a)}
\sim
\sum_{n\in\mathbb Z}\frac{e^{int}}{a-n},
\quad 0\le t\le2\pi,\ a\notin\mathbb{Z}.
\]
Let \(g_a\) denote the periodic extension of \(\frac{\pi e^{ia(t-\pi)}}{\sin(\pi a)}\).
The periodic function \(g_a\) is continuously differentiable away from the
points \(2\pi\mathbb Z\), and has finite one-sided limits at those points.
Hence, the Dirichlet--Jordan theorem \cite[Page 128]{SteinShakarchiFourier} implies that its symmetric Fourier
partial sums converge at every \(t\in\mathbb R\): they converge to \(g_a(t)\)
when \(t\notin2\pi\mathbb Z\), and to
\(
\frac{g_a(t-)+g_a(t+)}{2},
\)
when \(t\in2\pi\mathbb Z\).
Taking
\(
t=M(y+\pi),
\)
we conclude that for each fixed \(m\), the series on the right-hand side of
\eqref{pointwise_series} converges for every \(t\in(2l\pi,2(l+1)\pi)\), $l=0,1,\dots,M-1$. Then \eqref{gs0pre} holds for ever $y\in(y_{l},y_{l+1})$, where $y_l=-\pi+\frac{2l\pi}{M}$.

It remains to consider  points \(\{y_l\mid  l=1,2,\dots,M-1\}\). For every fixed \(m\),
the Dirichlet--Jordan theorem gives the midpoint value at \(y_\ell\).
Since the sum over \(m\) is finite and the factors
\(
\frac{F_\Lambda(j)e^{i\lambda_m y}}
     {M F_\Lambda'(\lambda_m)}
\)
are continuous in \(y\), it follows that
\[
\sum_{n\in\mathbb{Z}}
\sum_{m\in\mathbb{Z}_M}
\psi_{mn}(j)e^{i\tau_{mn}y_l}=\frac{\lim_{y\rightarrow{y_{l}^{+}}}e^{ijy}+\lim_{y\rightarrow{y_{l}^{-}}}e^{ijy}}{2}=e^{ijy_l}.
\]
Thus,
\[
e^{ijy}
=
\sum_{n\in\mathbb Z}\sum_{m\in\mathbb Z_M}
\psi_{mn}(j)e^{i\tau_{mn}y},
\quad y\in(-\pi,\pi).
\]
Finally, if
\(
j=\tau_{m_0n_0}
\)
for some \(m_0\in\mathbb Z_M\) and \(n_0\in\mathbb Z\), then the conclusion
follows directly from the cardinal interpolation property
\(
\psi_{mn}(\tau_{m_0n_0})=\delta_{m,m_0}\delta_{n,n_0}
\).
This completes the proof.

\subsection*{Appendix B. Proofs of Lemmas \ref{Cnorm} and \ref{ExampleCnorm}}\label{AppendixB}

\textbf{Proof of Lemma \ref{Cnorm}}. We use the standard identity
\[
\sum_{k\in\mathbb Z}\frac{1}{|k-x|^2}=\frac{\pi^2}{\sin^2(\pi x)}\mbox{ for any } x\notin\mathbb{Z}.
\]
Since $\min\{\mathsf{r}_1(\cdot),\mathsf{c}_1(\cdot)\}\le \mathsf{r}_1(\cdot)$, it is enough to bound $\mathsf{r}_1(\widetilde{\bC}_n)$.

\textbf{Case $n=0$.} For $m\in T$, we have $\delta_m:=|\lambda_m-m|\in(\frac14-\epsilon,\frac12]$. As $\epsilon\le\frac18$, we get $\delta_m>\frac18$. Hence,
\[
\sum_{j=0}^{M-1}\frac{1}{|j-\lambda_m|^2}
\le \sum_{k\in\mathbb Z}\frac{1}{|k-\lambda_m|^2}
= \frac{\pi^2}{\sin^2(\pi\delta_m)}.
\]
Since $\sin^2(\pi x)$ is increasing on $(0,\frac12]$, we have
\[
\mathsf{r}_1(\widetilde{\bC}_0)\le \sqrt{\frac{\pi^2}{\sin^2(\pi\delta_m)}}
<\frac{\pi}{\sin(\pi/8)}.
\]

\textbf{Case $n=\pm1$.} We begin with the case $n=1$. Fix $m\in T$ and set $\delta=\lambda_m-m$, with $|\delta|<1/2$. For $j\in\mathbb Z_M$,
\[
j-\lambda_m-M = j-m-M-\delta.
\]
Put $k=M+m-j$. As $j$ ranges over $0,1,\dots,M-1$, $k$ ranges over $m+1,m+2,\dots,M+m$, so $k\ge1$. Then
\[
|j-\lambda_m-M| = |k+\delta| \ge k-|\delta| > k-\frac12.
\]
Therefore,
\[
\sum_{j=0}^{M-1}\frac{1}{|j-\lambda_m-M|^2}
\le \sum_{k=1}^{\infty}\frac{1}{(k-\frac12)^2}
= \frac{\pi^2}{2}.
\]
The case $n=-1$ follows by a similar substitution (e.g., $k=j-m+M\ge1$). Hence,
\[
\mathsf{r}_1(\widetilde{\bC}_{\pm1}) \le \frac{\pi}{\sqrt2},
\]
and consequently $\min\{\mathsf{r}_1(\widetilde{\bC}_{\pm1}),\mathsf{c}_1(\widetilde{\bC}_{\pm1})\}\le \mathsf{r}_1(\widetilde{\bC}_{\pm1}) \le \pi/\sqrt2$. This proves the lemma.

\bigskip 

\textbf{Proof of Lemma \ref{ExampleCnorm}.}
For any $m\in T$ and $j\in\mathbb Z_M$, we have $|j-\lambda_m|\le M$ and $|j-v_m|\le M$. For $|n|\ge2$,
\[
|j-\lambda_m-nM| \ge |n|M - |j-\lambda_m| \ge (|n|-1)M,
\]
and similarly $|j-v_m-nM| \ge (|n|-1)M$. Thus,
\[
\Big|\frac{1}{j-\lambda_m-nM}\frac{j-v_m}{j-v_m-nM}(-1)^{nM}\Big|
\le \frac{1}{(|n|-1)M}\cdot \frac{M}{(|n|-1)M}
= \frac{1}{(|n|-1)^2M}.
\]
Summing over $j\in\mathbb Z_M$ yields the upper bound
\[
 M\cdot \frac{1}{(|n|-1)^4M^2}
= \frac{1}{(|n|-1)^4M}
\le \frac{1}{2(|n|-1)^4},
\]
as $M\ge2$. Taking the square root yields
\[
\mathsf{r}_1(\widetilde{\widetilde{\bC}}_n)\le \frac{1}{\sqrt{2}(|n|-1)^2}.
\]
Because $\min\{\mathsf{r}_1(\widetilde{\widetilde{\bC}}_{n}),\mathsf{c}_1(\widetilde{\widetilde{\bC}}_{n})\}\le \mathsf{r}_1(\widetilde{\widetilde{\bC}}_{n})$, the claimed bound follows.

\setcounter{section}{3}
\subsection*{Appendix C. Proof of Lemma \ref{L25}}\label{AppendixC}

Throughout the proof, we frequently use the elementary estimates \eqref{sincbound} for the sine function.

If $M=2$, then there is only one factor in \eqref{Prodeq}.
By the assumption
\[
\min_{k\in\mathbb Z_M,k\neq j}\operatorname{dist}_M(x,v_k)>\eta
\]
and \eqref{sincbound}, the left-hand side of \eqref{Prodeq} is bounded
by $1/\eta$.

If $M=3$, then the two points $v_k$, $k\ne j$, are separated modulo
$M$ by more than $1/2$. Hence, at least one of them has modular
distance at least $1/4$ from $x$. Since the other one has distance
larger than $\eta$, \eqref{sincbound} gives
\[
\Bigg|
\prod_{k\in\mathbb Z_M\setminus\{j\}}
\frac{\sin\frac{\pi(v_j-v_k)}{M}}
     {\sin\frac{\pi(x-v_k)}{M}}
\Bigg|
\lesssim \frac1\eta.
\]

We assume $M\ge4$ below. Set $y:=x-v_j$. By periodicity and
symmetry, we may choose the representative of $y$ such that
\[
0\le y=\operatorname{dist}_M(x,v_j)\le d\le\frac M2.
\]

For $k\in\mathbb Z_M$, let
\(p_k:=(j+k)\ {\textrm mod}\ M
\)
and choose $q_k\in\{0,1\}$ such that
$j+k=p_k+q_kM$. Define
\(
w_k:=q_kM+v_{p_k}-v_j.
\)
Then
$
w_k=k+(v_{p_k}-p_k)-(v_j-j).
$
In particular, $w_0=0$, and
\begin{equation}\label{wkeq1}
|w_k-k|<\frac12,
\quad
0<w_k<M,
\quad
k=1,2,\ldots,M-1.
\end{equation}
Moreover,
\[
\operatorname{dist}_M(w_k,w_\ell)
=
\operatorname{dist}_M(v_{p_k},v_{p_\ell})
>\frac12,
\quad k\ne\ell,
\]
and
\begin{equation}\label{wkyeta}
\operatorname{dist}_M(w_k,y)
=
\operatorname{dist}_M(v_{p_k},x)
>\eta.
\end{equation}

Using the parity and periodicity of the sine function, we obtain
\begin{equation}\label{twoparts}
\Bigg|
\prod_{k\in\mathbb Z_M\setminus\{j\}}
\frac{\sin\frac{\pi(v_j-v_k)}{M}}
     {\sin\frac{\pi(x-v_k)}{M}}
\Bigg|
=
\prod_{k=1}^{M-1}
\Bigg|
\frac{\sin\frac{\pi w_k}{M}}
     {\sin\frac{\pi(w_k-y)}{M}}
\Bigg|.
\end{equation}

For brevity, write
\[
S(t):=\Big|\sin\frac{\pi t}{M}\Big|.
\]
By \eqref{sincbound},
\begin{equation}\label{Sbound}
\frac{2}{M}\operatorname{dist}_M(t,0)
\le S(t)
\le
\frac{\pi}{M}\operatorname{dist}_M(t,0).
\end{equation}

We first consider $0\le y<3/2$. For $k=1,2$, we have
$S(w_k)\lesssim M^{-1}$. Since $w_1,w_2$ are separated modulo
$M$ by more than $1/2$, at most one of them can have distance less
than $1/4$ from $y$. Using \eqref{wkyeta} for this possible point
and \eqref{Sbound} for the other one, we obtain
\begin{equation}\label{small-local}
\prod_{k=1}^{2}
\frac{S(w_k)}{S(w_k-y)}
\lesssim\frac1\eta.
\end{equation}
For $k\ge3$, \eqref{wkeq1} gives $w_k>y$. The function
\(
t\mapsto \sin\frac{\pi t}{M}/
     \sin\frac{\pi(t-y)}{M}
\)
is decreasing on $(y,M)$. Therefore,
\[
\prod_{k=3}^{M-1}\frac{S(w_k)}{S(w_k-y)}
\le
\prod_{k=3}^{M-1}
\frac{S(k-\frac12)}{S(k-\frac12-y)}.
\]
Put $\beta=y+\frac12$. The finite product identity \cite[(548) in Section 1.391]{Gradshteyn2014}
\[
\prod_{k=0}^{M-1}
\Big|\sin\frac{\pi(k+a)}{M}\Big|
=
2^{1-M}|\sin(\pi a)|
\]
gives
\[
\prod_{k=3}^{M-1}
\frac{S(k-\frac12)}{S(k-\frac12-y)}
=
\frac{
S(\beta)S(1-\beta)S(2-\beta)}
{
|\sin(\pi\beta)|S(\frac12)^2S(\frac32)
},
\]
where the quotient is understood by continuity at integer values
of $\beta$.

Let $\delta=\operatorname{dist}(\beta,\mathbb Z)$. Then
$|\sin(\pi\beta)|\gtrsim\delta$, while one of the three factors
in the numerator is $\lesssim\delta/M$ and the other two are
$\lesssim M^{-1}$. Also,
\(
S\Big(\frac12\Big)^2S\Big(\frac32\Big)
\gtrsim M^{-3}.
\)
Hence
\[
\prod_{k=3}^{M-1}\frac{S(w_k)}{S(w_k-y)}
\lesssim1.
\]
Together with \eqref{small-local}, this proves \eqref{Prodeq} for
$0\le y<3/2$.

We now assume $y\ge3/2$. Write
\[
y=m-\frac12+\alpha,
\quad
m\ge2,
\quad
0\le\alpha<1.
\]
Since $y\le M/2$, we have $m\le(M+1)/2$. Let
\(
\mathcal L:=\{m-1,m,m+1\}.
\)
For $k\in\mathcal L$, \eqref{wkeq1} gives
\(
S(w_k)\lesssim\frac{1+y}{M}.
\)
Since the $w_k$ are mutually separated modulo $M$ by more than
$1/2$, at most one of them can have distance less than $1/4$ from
$y$. Using \eqref{wkyeta} for this possible point and
\eqref{Sbound} for the other two, we obtain
\begin{equation}\label{local-bound}
\prod_{k\in\mathcal L}
\frac{S(w_k)}{S(w_k-y)}
\lesssim
\frac{(1+y)^3}{\eta}.
\end{equation}

For $k\le m-2$, the function
\(
t\mapsto
\sin\frac{\pi t}{M}/
     \sin\frac{\pi(y-t)}{M}
\)
is increasing on $(0,y)$, whereas for $k\ge m+2$ the function
\(
t\mapsto\sin\frac{\pi t}{M}/\sin\frac{\pi(t-y)}{M}
\)
is decreasing on $(y,M)$. Consequently,
\[
\begin{aligned}
&\prod_{k\in\{1,2,\ldots,M-1\}\setminus\mathcal L}
\frac{S(w_k)}{S(w_k-y)}
\le
\prod_{k=1}^{m-2}
\frac{S(k+\frac12)}{S(k+\alpha)}
\prod_{k=m+1}^{M-2}
\frac{S(k+\frac12)}{S(k+\alpha)}.
\end{aligned}
\]
Using the same finite product identity, the right-hand side equals
\[
\frac{S(\alpha)S(1-\alpha)}
     {|\sin(\pi\alpha)|S(\frac12)^2}
\,
\frac{S(m-1+\alpha)S(m+\alpha)}
     {S(m-\frac12)S(m+\frac12)},
\]
where the expression is understood by continuity at $\alpha=0$.

By \eqref{Sbound},
\(
\frac{S(\alpha)S(1-\alpha)}
     {|\sin(\pi\alpha)|S(\frac12)^2}
\lesssim1.
\)
Moreover, each numerator argument in the second quotient differs
from the corresponding denominator argument by at most $1/2$,
while $\operatorname{dist}_M(m\pm \frac12,0)\ge1$ due to $m\ge2$.
Another application of \eqref{Sbound} therefore gives
\[
\frac{S(m-1+\alpha)S(m+\alpha)}
     {S(m-\frac12)S(m+\frac12)}
\lesssim1.
\]
Thus
\begin{equation}\label{nonlocal-bound}
\prod_{k\in\{1,2,\ldots,M-1\}\setminus\mathcal L}
\frac{S(w_k)}{S(w_k-y)}
\lesssim1.
\end{equation}

Combining \eqref{twoparts}, \eqref{local-bound}, and
\eqref{nonlocal-bound}, we obtain
\[
\Bigg|
\prod_{k\in\mathbb Z_M\setminus\{j\}}
\frac{\sin\frac{\pi(v_j-v_k)}{M}}
     {\sin\frac{\pi(x-v_k)}{M}}
\Bigg|
\lesssim
\frac{(1+y)^3}{\eta}.
\]
Since $0\le y\le d$, the desired result follows.

\setcounter{section}{4}
\subsection*{Appendix D. Proof of Theorem \ref{Thmlowerbound}}\label{AppendixD}

We estimate a lower bound for the spectral norm of the corresponding periodic nonuniform interpolation matrix $\bB_{\sigma,L}$ defined as in \eqref{DefeqBsigmaL} by choosing a special frequency set \(\Lambda=\{\lambda_k\mid  k\in\mathbb{Z}_M\}\). We require the following auxiliary lemma.

\begin{lemma}\label{445} Let $N\ge1$ be an integer and take $M=2N+1$. 
Let $1/4\le L<1/2$ and let \(\Lambda=\{\lambda_k\mid  k\in\mathbb{Z}_M\}\), where $\lambda_k$ follows the specification in \eqref{eqspecifiedLambda}. The function $F_{\Lambda}$ is defined by \eqref{FLambdax}.
Then, 
\[
 \Big|
\sum_{n\in\mathbb Z}
\frac{e^{-inM/10}F_\Lambda(0)}
{F'_\Lambda(\lambda_N+nM)(-\lambda_N-nM)}
\Big|
\gtrsim
\frac{M^{4L-1}}{1-2L}.
\]
\end{lemma}
\begin{proof}
Since $M=2N+1$, we have \(\lambda_N=N+L\). Set
\(b:=-\lambda_N=-(N+L)\). Because $M$ is odd, the function $F_{\Lambda}$ is anti-periodic with period $M$, namely, 
\(F_\Lambda(x+M)=-F_\Lambda(x)\).
Consequently, 
\begin{equation}\label{Lowerboundeq1}
\sum_{n\in\mathbb Z}
e^{-inM/10}
\frac{F_\Lambda(0)}
{F'_\Lambda(\lambda_N+nM)(-\lambda_N-nM)}
=
\frac{F_\Lambda(0)}{F'_\Lambda(\lambda_N)}
\sum_{n\in\mathbb Z}
\frac{e^{-inM/10}(-1)^n}{b-nM}.
\end{equation}
Let $\theta_M\in(-\pi,\pi)$ such that $\theta_M:=M/10\pmod{2\pi}$. Invoking \eqref{Fourierexpansion} with $a=b/M$ and $t=\pi-\theta_M$, we obtain
\[
\sum_{n\in\mathbb Z}
\frac{e^{-inM/10}(-1)^n}{b-nM}
=-\frac{1}{M}\sum_{n\in\mathbb Z}
\frac{e^{in(\pi-M/10)}}{n-b/M}
=-\frac{\pi}{M}
e^{i\frac{b}{M}(\theta_M-\pi)}
\frac{1}{\sin(\frac{\pi b}{M})}.
\]
Therefore,
\begin{equation}\label{Lowerboundeq2}
\Big|
\sum_{n\in\mathbb Z}
\frac{e^{-inM/10}(-1)^n}{b-nM}
\Big|
=
\frac{\pi}{M}
\Big|\frac{1}{\sin(\frac{\pi b}{M})}
\Big|\in \Big[\frac{\pi}{M},
\frac{\sqrt{2}\pi}{M}\Big],
\end{equation}
where we have used the fact $1\ge |\sin(\pi b/M)|\ge 1/\sqrt{2}$ as
\(\frac{b}{M}=-\frac{N+L}{2N+1}\in(-\frac12,-\frac14)\).

Therefore, it suffices to prove the lower bound
\[
\Big|
\frac{F_\Lambda(0)}{F_\Lambda'(\lambda_N)}
\Big|
\gtrsim
\frac{M^{4L}}{1-2L}.
\]
Observe that $\Lambda=\{j+L\mid  0\le j\le N\}\cup\{j-L\mid  N+1\le j\le 2N\}$.
We factor
\[
F_\Lambda(x)=A(x)H(x),
\]
where
\[
A(x)
=
\prod_{j=0}^{M-1}
\sin\Big(\frac{\pi(x-j-L)}{M}\Big)
\mbox{ and }
H(x)
=
\prod_{j=N+1}^{2N}
\frac{
\sin\big(\frac{\pi(x-j+L)}{M}\big)
}{\sin\big(\frac{\pi(x-j-L)}{M}\big)}.
\]

The finite product identity for the sine function gives
\(A(x)=(-1)^{M+1} 2^{1-M}\sin\pi(x-L)\) (see, for example, \cite[(1.392)]{Gradshteyn2014}).
Since $\lambda_N=N+L$ is a zero of $A$, we have
\[
F_\Lambda'(\lambda_N)=A'(\lambda_N)H(\lambda_N)+A(\lambda_N)H'(\lambda_N)=A'(\lambda_N)H(\lambda_N)
\]
and hence
\begin{equation}\label{L51eq0}
\Big|
\frac{F_\Lambda(0)}{F_\Lambda'(\lambda_N)}
\Big|
=
\frac{|A(0)|}{|A'(\lambda_N)|}
\Big|
\frac{H(0)}{H(\lambda_N)}
\Big|.
\end{equation}
Direct computation yields \(|A(0)|=2^{1-M}\sin(\pi L)\) and \(|A'(\lambda_N)|=\pi 2^{1-M}.\)
Since $1/4\le L<1/2$ and $\sin(\pi L)\ge 1/\sqrt{2}$,
\begin{equation}\label{L51eq00}
\frac{|A(0)|}{|A'(\lambda_N)|}\ge\frac{1}{\pi\sqrt{2}}.
\end{equation}

It remains to estimate $H$. Define
\(
S(t):=|\sin(\frac{\pi t}{M})|.
\)
Writing $j=N+r$, $1\le r\le N$, and using the symmetry $S(t)=S(M-t)$, we obtain
\[
|H(0)|
=
\prod_{r=1}^{N}
\frac{S(N+r-L)}{S(N+r+L)}=
\prod_{r=1}^{N}
\frac{S(r+L)}{S(r-L)}.
\]
Likewise,
\[
|H(\lambda_N)|
=
\prod_{r=1}^{N}
\frac{S(r-2L)}{S(r)}.
\]
Therefore
\[
\Big|
\frac{H(0)}{H(\lambda_N)}
\Big|
=
\prod_{r=1}^{N}
\frac{S(r+L)}{S(r-L)}
\prod_{r=1}^{N}
\frac{S(r)}{S(r-2L)}.
\]

For all $r=1,2,\dots,N$, the arguments $r+L$, $r-L$, $r$, and $r-2L$ lie in $(0,M/2)$. Since $\frac{1}{\sin t}\le 1+\frac{1}{t}$ for any $0<t\le \frac{\pi}{2}$ (as the function $(1+t)\sin(t) -t$ is strictly increasing on $(0,\frac{\pi}{2}]$ with  a function value $0$ at $t=0$),
\begin{equation}\label{L51eq1}
\prod_{r=1}^{N}\frac{S(r+L)}{S(r-L)}
\asymp
\prod_{r=1}^{N}\frac{r+L}{r-L}=\frac{\Gamma(N+1+L)\Gamma(1-L)}{\Gamma(N+1-L)\Gamma(1+L)}
\asymp N^{2L},
\end{equation}
where the last comparison follows from the Gautschi inequality \cite[(5.6.4)]{Olver2010}
\[
x^{1-s}<\frac{\Gamma(x+1)}{\Gamma(x+s)}<(x+1)^{1-s},\mbox{ for any }x>0, 0<s<1.
\]
Similarly,
\[
\prod_{r=1}^{N}\frac{S(r)}{S(r-2L)}
\asymp
\prod_{r=1}^{N}\frac{r}{r-2L}=
\frac{\Gamma(N+1)\Gamma(1-2L)}{\Gamma(N+1-2L)}.
\]
Again, the Gautschi inequality gives $\Gamma(N+1)/\Gamma(N+1-2L)\asymp N^{2L}$. As $0<1-2L\le 1/2$ and
\[
\Gamma(x)=\int_0^{\infty} t^{x-1}e^{-t}dt\ge\int_0^1 t^{x-1}e^{-t}dt
\ge e^{-1}\int_0^1 t^{x-1}dt=\frac{e^{-1}}{x},
\]
we have $\Gamma(1-2L)\gtrsim 1/(1-2L)$. Hence
\begin{equation}\label{L51eq2}
\prod_{r=1}^{N}\frac{s(r)}{s(r-2L)}
\gtrsim
\frac{N^{2L}}{1-2L}.
\end{equation}
Combining \eqref{L51eq1} and \eqref{L51eq2} yields
\[
\Big|\frac{H(0)}{H(\lambda_N)}\Big|
\gtrsim
\frac{N^{4L}}{1-2L}.
\]
By \eqref{L51eq0}, \eqref{L51eq00}, and $M=2N+1$, 
\begin{equation}\label{Lowerboundeq3}
\Big|\frac{F_\Lambda(0)}{F_\Lambda'(\lambda_N)}\Big|
\gtrsim
\frac{M^{4L}}{1-2L}.
\end{equation}
Combining \eqref{Lowerboundeq1}, \eqref{Lowerboundeq2}, and  \eqref{Lowerboundeq3}, we obtain
\[
\Big|
\sum_{n\in\mathbb{Z}}
\frac{e^{-inM/10}F_\Lambda(0)}
{F_\Lambda'(\lambda_N+nM)(-\lambda_N-nM)}
\Big|
\gtrsim
\frac{M^{4L}}{1-2L}\cdot \frac1M
=
\frac{M^{4L-1}}{1-2L},
\]
which yields the desired conclusion.
\end{proof}

\textbf{Proof of Theorem \ref{Thmlowerbound}.} 

Recalling that $\psi_{mn}^{*}$ is defined as in \eqref{psimnstar}, we have
\[
\psi^*_{Nn}(0)=\frac{F_{\Lambda}(0) e^{-inM/10}}{F_{\Lambda}'(\tau_{Nn})(-\lambda_N-nM)}.
\]
By Lemma \ref{LemmaBsigmaL} and the definition of $\bB_{\sigma,L}$ defined as in \eqref{DefeqBsigmaL},
\begin{equation}\label{Lowerboundeq4}
\|\bA_{\sigma,L}^{-1}\|_2\ge M^{-\frac12}\|\bB_{\sigma,L}\|_2\ge M^{-\frac12}\Big|\sum_{n\in\mathbb{Z}}\psi_{Nn}^*(0)\Big|,
\end{equation}
where we have used the fact that the spectral norm of a matrix is at least the absolute value of any entry. Combining Lemma \ref{445} and \eqref{Lowerboundeq4}, we obtain
\[
\|\bA_{\sigma,L}^{-1}\|_2\gtrsim M^{-\frac12}\frac{M^{4L-1}}{1-2L}=\frac{M^{4L-\frac32}}{1-2L},
\]
which completes the proof of Theorem \ref{Thmlowerbound}.

\section*{Declaration of generative AI use}
We used DeepSeek and ChatGPT-5.5 Plus to assist with language translation and polishing, as well as with the preparation of BibTeX entries. These tools were also used to generate initial proof drafts for Lemma \ref{L25}, Lemma \ref{L31}, and Lemma \ref{L32}, based on mathematical statements of the lemmas provided by the authors.
In particular, the authors developed the statements and auxiliary constructions of Lemma \ref{L31} and Lemma \ref{L32}, together with their roles in the proof of Theorem \ref{T001}. Once the auxiliary constructions had been formulated by the authors, the remaining arguments consisted mainly of direct estimates. The AI-generated proof drafts were subsequently checked, corrected, and revised by the authors. The authors take full responsibility for the mathematical correctness, originality, and final content of the manuscript.

\begin{funding}
Liang Chen was supported in part
by the National Natural Science Foundation of China under grant 12461019. Haizhang Zhang was supported in part by the National Natural Science Foundation of China under grant 12371103.
\end{funding}


\bibliographystyle{emss}
\bibliography{sample}

\end{document}